%% file: main.tex
\documentclass[letterpaper, 10 pt, conference]{ieeeconf}

\input{package}

\input{notation}

\begin{document}

\title{Robust Control Lyapunov-Value Functions for Nonlinear Disturbed Systems} 


\author{Zheng Gong ,\IEEEmembership{Student Member, IEEE}, and Sylvia Herbert, \IEEEmembership{Member, IEEE}
\thanks{This work is supported by ONR YIP (\#N00014-22-1-2292) and the UCSD JSOE Early Career Faculty Award.} 
\thanks{All authors are in Mechanical and Aerospace Engineering at UC San Diego (e-mail: {\href{mailto:zhgong@ucsd.edu}{zhgong}, \href{mailto:sherbert@ucsd.edu}{sherbert}\}@ucsd.edu.})
}
}


\maketitle



\begin{abstract}                          
This article presents a method to construct robust control Lyapunov value functions (R-CLVFs) for robust and stabilizing control of nonlinear systems with input and disturbance bounds.  
Through modifications to Hamilton-Jacobi reachability analysis, the R-CLVF can be computed via dynamic programming. We prove that the R-CLVF can be used to stabilize the system to the smallest robust control invariant set of any point of interest at a user-specified exponential rate. The R-CLVF additionally can be used to obtain the domain over which stabilization at the desired rate is possible, i.e., the region of exponential stabilizability. Due to the computationally expensive nature of dynamic programming, we additionally propose methods to improve tractability via decomposing the dynamical model or warm-starting the computation. Under certain assumptions, we show that these approaches maintain exact solutions.
Three numerical examples are provided, validating our definition of the smallest robust control invariant set, illustrating the impact of the exponential rate and different loss functions, and demonstrating the efficiency of computation using warm-starting and decomposition.
\end{abstract}




\section{INTRODUCTION}
\label{sec: introduction}
\input{Introduction}
\section{BACKGROUND}
\label{sec: background}
\input{Background}
\section{ROBUST CONTROL LYAPUNOV-VALUE FUNCTIONS}
\label{sec: r-clvf}
{In this section, we will pose an optimal control problem, whose value function is a viscosity solution to an HJ-type VI. This VI can be used to prove that the value function has CLF properties, i.e., it implies exponential stabilizability. We highlight that we use an exponential amplifier in our value function instead of a discount factor (which is common in the literature), and that our formulation has very straightforward intuition. }
This section is organized as follows: Section~\ref{Sec:TV-R-CLVF} defines the TV-R-CLVF and shows some key properties of it. Section~\ref{Sec:R-CLVF} defines the R-CLVF, shows the sufficient condition for its existence, and shows it satisfies the DPP (Theorem~\ref{thm:CLVF_dynamic_programming_principle}) and is the viscosity solution of a VI (Theorem~\ref{thm:CLVF_VI_viscosity_solution}). {Section~\ref{Sec:R-CLVF-EXP} shows that the existence of the R-CLVF (Theorem~\ref{thrm:CLVF_finite_eq_EC}) is equivalent to the exponential stabilizability to the SRCIS.} 
Section~\ref{Sec:RCLVF-QP} provides an feasibility guaranteed QP controller. 


\subsection{TV-R-CLVF} \label{Sec:TV-R-CLVF}
\begin{definition} A TV-R-CLVF is a function $V_{\gamma}(\state,\tinit): \mathbb R^n \times \mathbb R_- \rightarrow \mathbb R $ defined as: 
 \begin{align} \label{eqn:finite_time_CLVF}
    \tclvf (\state,\tinit) = \sup _{ \dmap  \in \Dmap_\tinit} \inf _{\csig \in \cfset_\tinit} J_\gamma ( \tinit ,\state, \csig, \dmap),
\end{align} 
where $J_\gamma ( \tinit ,\state, \csig, \dsig)$ is the cost function:
 \begin{align}\label{eqn:finite_time_CLVF_loss}
    J_\gamma ( \tinit ,\state, \csig, \dsig)  =\max _{ \tvar \in [\tinit, \thor]} e^{\gamma( \tvar - \tinit)}  \loss \bigl( \xi(\tvar; \tinit ,\state, \csig, \dsig) \bigl),
\end{align} 
$\gamma \geq 0$ is a user-specified parameter that represents the desired decay rate, and $ \loss (\state;\poi) = h(\state;\poi) - \minval$. 
\end{definition}

The cost at a state captures the maximum exponentially amplified distance between the trajectory (starting from this state) and the zero-level set of $\loss(\state)$. The optimal control tries to minimize this cost and seeks to drive the system towards the POI. In contrast, the disturbance tries to maximize the cost and push the system away from the POI. 

From~\eqref{eqn:finite_time_CLVF} and~\eqref{eqn:finite_time_CLVF_loss}, if a trajectory is initialized outside the zero sub-level set of $\loss(\state)$, the value $\tclvf (\state,\tinit)$ will always be positive. If starting inside the zero sub-level set of $\loss(\state)$, the value $\tclvf (\state,\tinit)$ might be positive or negative, depending on the time horizon $\tinit$. In other words, as $\tinit$ decreases, the zero sub-level set of $\tclvf (\state,\tinit)$ will shrink. In fact, we will show later that, given $\tinit$, the zero sub-level set of $\tclvf (\state,\tinit)$ is exactly the zero sub-level set of $V(\state, \tinit) - \minval$ (Lemma~\ref{lemma:same_zero_level}).

It should be noted that $\gamma$ serves as an exponential \textbf{amplifier} for the TV-R-CLVF, which in turn sets the desired exponential \textbf{convergence} rate of a trajectory to the SRCIS. This is introduced more formally in  Lemma~\ref{lemma:CLVF_exists}. 

We now prove some properties of the TV-R-CLVF in Proposition~\ref{prop: TV-R-CLVF_lipschitz} and show the mathematical foundation for how it can be obtained in Theorem~\ref{thrm:TV-R-CLVF_DPP} and~\ref{thm:TV-R-CLVF-VI-VS}. 
\begin{proposition} \label{prop: TV-R-CLVF_lipschitz}
    The TV-R-CLVF is Lipschitz in $\state, \tinit$ and bounded for any open set $\mathcal C$ .

    \begin{proof}
    See the Appendix.~\ref{App: proof_prop1}.
    \end{proof}
\end{proposition}

We now show that the TV-R-CLVF satisfies the DPP and is the unique viscosity solution to the TV-R-CLVF-VI. {Theorem~\ref{thrm:TV-R-CLVF_DPP} forms the foundation for the proof of Theorem~\ref{thm:CLVF_dynamic_programming_principle}; therefore, its proof is included in the main text of the paper. }

\begin{theorem}\label{thrm:TV-R-CLVF_DPP}
    $V_\gamma(\state,\tinit)$ satisfies the following dynamic programming principle for all $t<t+\delta \leq 0$: 
    \begin{align}\label{eqn:TV-R-CLVF_DPP}
       & V_\gamma(\state, \tinit) = \sup _{\dmap \in \Dmap_\tinit} \inf _{u \in \mathbb{U}} \max \biggl\{e^{\gamma \delta} V_\gamma(\xi(t+\delta),t+\delta), \notag \\
    & \hspace{6em }\max_{s\in[t, t+\delta]} e^{\gamma (s-t)} \loss(\xi(s)) \biggl\}.
    \end{align}
    \input{proofs/proof_thrm_TV-DPP}
\end{theorem}

\begin{theorem}\label{thm:TV-R-CLVF-VI-VS} The  TV-R-CLVF is the unique viscosity solution to the following VI, 
   \begin{equation} \label{eqn:TV-R-CLVF-VI}
     \begin{aligned}
         &\max\biggl\{\loss(x) - V_\gamma (\state,t), \\
         &\hspace{.1cm} D_\tinit V_\gamma+ \max_{\dstb \in \dset } \min_{ \ctrl \in \cset} D_\state V_\gamma \cdot \dyn(\state, \ctrl, \dstb) + \gamma V_\gamma  \biggl\} =0,
    \end{aligned} 
\end{equation} 
with \textbf{terminal} condition $V_\gamma(\state, \thor) = \loss(\state)$. 

    \begin{proof}
        See the Appendix~\ref{App: proof_thrm2}.
    \end{proof}

\end{theorem}

{The proof of Theorem~\ref{thm:TV-R-CLVF-VI-VS} is analogous to that of Theorem~\ref{thm:CLVF_VI_viscosity_solution}. Since the main focus of this paper is Theorem~\ref{thm:CLVF_VI_viscosity_solution}, we include the proof of Theorem~\ref{thm:TV-R-CLVF-VI-VS} in the Appendix.} It should be noted that the definitions of viscosity solutions for the terminal value problem (TVP) and initial value problem (IVP) are slightly different. For more details, the reader is referred to~\cite{crandall1983viscosity,crandall1984some} for IPV and~\cite{barron1989bellman} for TVP.  

With Theorems~\ref{thrm:TV-R-CLVF_DPP} and~\ref{thm:TV-R-CLVF-VI-VS}, the TV-R-CLVF can be obtained by solving equation~\eqref{eqn:TV-R-CLVF-VI} for $\tvar \in [\tinit,\thor]$. Note that in equation~\eqref{eqn:TV-R-CLVF-VI}, define $H:\mathbb R^n  \times \mathbb {R} ^n \mapsto \mathbb R$, 
\begin{align} \label{eqn:Hamitonian}
    H(x,p) = \max_{\dstb \in \dset } \min_{ \ctrl \in \cset} p \cdot \dyn(\state, \ctrl, \dstb).
\end{align}
$H$ is called the Hamiltonian, and the compactness of $\cset$ and $\dset$ and Lipschitz continuity of $\dyn$ guarantee that $H$ is bounded and Lipschitz continuous with respect to $(x,p)$.


\subsection{R-CLVF} \label{Sec:R-CLVF}
{In this section, we introduce the R-CLVF. We will show one sufficient condition for the existence of the R-CLVF, and show some key properties of it. Finally, we will also show that the R-CLVF is a viscosity solution to a VI. }

\begin{definition}\textbf{(R-CLVF)} Given an open set $D_\gamma \subseteq \mathbb R^n$ and $\gamma > 0$, the function $\clvf : \mathcal D_\gamma \mapsto \mathbb R$ is a R-CLVF if the following limit exists: 
 \begin{align} \label{eqn:infinite_time_CLVF}
   \clvf (\state) = \lim _{\tinit \rightarrow -\infty} \tclvf (\state,\tinit).
\end{align} 
\end{definition}
Notice the domain of the R-CLVF is $\dom$, where $\dom$ can be $\mathbb R^n$ if the limit exists on $\mathbb R^n$, {in contrast, for the TV-R-CLVF, the domain is always $\mathbb{R}^n$}. Also notice that for the R-CLVF, $\gamma$ is strictly greater than 0. The reason is that if $\gamma = 0$, the limit in~\eqref{eqn:infinite_time_CLVF} may exist on a closed set. 
\begin{remark} \label{remark:counter_ex}
    Note that though in Definition 4 from \cite{gong2022constructing}, we claimed the domain is a compact set, this is incorrect. As a result, the proof in Remark 1 in~\cite{gong2022constructing} is also incorrect, and we can show that the convergence is \textbf{not} uniform. To see these, consider a 1D toy example $\dot x = x+u$, where $u \in [-1,1]$ and $\gamma > 0$. Then, for any trajectory starting from $ x \in (-1,1)$, there exists a control signal so that the trajectory reaches $0$ at some finite time, therefore the limit in~\eqref{eqn:infinite_time_CLVF} exists. For trajectories starting from $x = \pm 1$, with best control efforts, they can only stay at the same point, so the limit in~\eqref{eqn:infinite_time_CLVF} does not exist. For trajectories starting from $x>1 $ or $x<-1$, the trajectory will diverge, and the limit in~\eqref{eqn:infinite_time_CLVF} does not exist. By taking feedback 
    \begin{align*}
        u^*(\state) =
            1 \text{ if } x< 0, -1 \text{ if }  x> 0,  0 \text{ if } x = 0,
    \end{align*}
    $\clvf(\state)$ is attained at $t = \min \{0,-\ln\frac{\gamma}{(1+\gamma)(1-|x|)}\}$, and as $x\rightarrow \pm 1$, $t \rightarrow -\infty$, so the convergence is not uniform. Further, we could compute the maximum value is actually given by $|x |$ if $t = 0$ and $\frac{\gamma^\gamma}{(1-|x|)^\gamma(\gamma+1)^\gamma}(1-\frac{\gamma}{1+\gamma})$ if $t = -\ln\frac{\gamma}{(1+\gamma)(1-|x|)}$. Therefore, as $x\rightarrow \pm 1$, $\clvf(\state) \rightarrow \infty$. On the other hand, when $\gamma = 0$, for trajectories starting from $[-1,1]$, they could at least stay at their initial state, and the value is $\clvf(\state) = |\state|$, which is attained at the initial time. Therefore, the convergence is uniform, and the domain is closed. To avoid this situation, we reiterate that for R-CLVF, $\gamma > 0$. 
\end{remark}

{For different choices of $\gamma$, do we obtain completely different R-CLVFs? Although the function values can vary significantly with $\gamma$, we can show that for any $\gamma>0$, the zero sub-level set remains unchanged.} { Denote the zero sub-level sets of TV-R-CLVF and R-CLVF as \begin{align*} 
    \mathcal Z_\gamma(\tinit):=\{ \state: \tclvf (\state,\tinit) \leq 0 \}, \hspace{1em}
    \mathcal Z_\gamma^\infty:=\{ \state: \clvf ( \state) \leq 0 \}.
\end{align*}}
{\begin{lemma} \label{lemma:same_zero_level}
    For all $\gamma > 0$, $\mathcal Z_\gamma(\tinit)$ are the same. Further, $\mathcal Z_\gamma^\infty$ are also the same and $\mathcal Z_\gamma^\infty = \mrcis$. 
    \begin{proof}
        \input{proofs/proof_sameZeroSet}

    \end{proof}
\end{lemma}}

{This lemma links the R-CLVF to the SRCIS by showing that the SRCIS is exactly the zero sublevel set of the R-CLVF. Since our goal is to prove that an R-CLVF implies exponential stabilizability to the SRCIS with rate $\gamma$, it is essential that the R-CLVF encode the SRCIS in a $\gamma$-independent way. In particular, although the SRCIS itself does not depend on $\gamma$, the R-CLVF values do. We therefore establish that for any $\gamma > 0$, the R-CLVF shares the same zero sublevel set, and that set is exactly the SRCIS.} 

{Since the R-CLVF is defined as the limit function of TV-R-CLVF, it is natural to ask under what conditions the limit exists. Consider an initial state $x$ in the SRCIS, then $\loss(\state) < 0$. Then, for its value to be $0$, $\loss(\traj)$ has to converge to $0$ with an exponential rate $\gamma$, which cannot be guaranteed. Therefore, we have $\clvf(x)\leq 0$ for all $x\in \mrcis$, regardless of $\gamma$. This means the limit always exists in the SRCIS. }

{For the states outside $\mrcis$, the value can be infinity because of the exponential amplifier $\gamma$. The next Lemma shows that the exponential stabilizability to SRCIS is a sufficient condition for the existence of the R-CLVF. }

\begin{lemma}[Existence of R-CLVF] \label{lemma:CLVF_exists}
    {The R-CLVF exists on an open set $\dom$ (or $\mathbb R^n$) if the system is exponentially stabilizable (under rate $\gamma$) to its SRCIS from $\roes$ (or $\mathbb R^n$), despite worst-case disturbance. Further $\roes \subset \dom$. }
    
\input{proofs/proof_ExpimplyExists}
\end{lemma}

{This Lemma shows that for nonlinear systems, the existence of the R-CLVF can be justified by analyzing the system dynamics. This sufficient condition is intuitive.} The R-CLVF value of a state $\state$ captures the maximum exponentially amplified distance between the optimal trajectory (starting from this state) and the zero-level set of $\loss(x)$. For this value to be finite, the optimal trajectory must converge to the SRCIS under an exponential rate no slower than $\gamma$. {In Section~\ref{Sec:R-CLVF-EXP}, we will show that this is also a necessary condition. }

\begin{proposition} \label{Prop:R-CLVF-lipschitz}
    The R-CLVF is locally Lipschitz continuous for all $x \in \dom$, if $\gamma \leq  \alpha$, where $\alpha$ is the fastest rate of exponential stabilizability of the system. 
    
    \begin{proof}
        See the Appendix~\ref{App: proof_prop2} .
    \end{proof}
\end{proposition}

Further, it {can be seen} that if the domain $\dom$ is $\mathbb R^n$, then R-CLVF is radially unbounded, i.e. $\lim_{\| \state \| \rightarrow \infty} \clvf(\state) = \infty$. This is because the R-CLVF is lower bounded by a radially unbounded function $\loss$. When the domain is an open subset of $\mathbb R^n$, we have $\lim_{\state  \rightarrow \partial \dom} \clvf(\state) = \infty$. 

\begin{proposition} \label{prop: Radially_unbounded}
When $\dom $ is an open subset of $ \mathbb R^n$, $\lim_{\state  \rightarrow \partial \dom} \clvf(\state) = \infty$. 
    
    \begin{proof}
        See the Appendix~\ref{App: proof_prop3}.
    \end{proof}
\end{proposition}

We now show that the R-CLVF satisfies the DPP, and is the viscosity solution to a VI. These two Theorems are crucial for the proof of later results. The proofs of these two Theorems are considered the main contributions of the paper and provide a means for constructing the R-CLVF for general nonlinear systems with {compact} control and disturbance.

\begin{theorem} \label{thm:CLVF_dynamic_programming_principle}
\textbf{(R-CLVF-DPP)} For all $ \tinit \leq \tinit +\delta \leq 0$, the following is satisfied 
\begin{align}  \label{eqn:CLVF_DPP} 
   \clvf(\state) = \sup _{\dmap \in \Dmap_\tinit} \inf _{u \in \mathbb{U}} \max \biggl\{e^{\gamma \delta} V^\infty_\gamma(z), \hspace{5mm} \notag \\
    \max_{ \tvar \in [\tinit,\tinit+\delta]} e^{\gamma (\tvar-\tinit)} \loss(\xi(\tvar ;\tinit,x,u,\dmap [u])) \biggl\}.
\end{align}  
\input{proofs/proof_thrm_DPP}

\end{theorem} 

\begin{theorem}\label{thm:CLVF_VI_viscosity_solution} \textbf{(R-CLVF-VI)} The R-CLVF is {a} solution to the following R-CLVF-VI in the viscosity sense, 
\begin{equation} \label{eqn:CLVF-VI}
     \begin{aligned}
         &\max\biggl\{\loss(x) - \clvf(x), \\
         &\hspace{.1cm} \max_{d \in \dset} \min_{ u \in \cset} D_x \clvf \cdot \dyn(\state, \ctrl, \dstb) + \gamma \clvf(\state) \biggl\} =0.
    \end{aligned} 
\end{equation}   
\input{proofs/proof_infiniteVS} 
\end{theorem}

However, it should be noted that the R-CLVF-VI~\eqref{eqn:CLVF-VI} may have multiple solutions given different choices of $\gamma$. To see this, let's see the 1D system $\dot x = u$, where $u \in [-2,2]$. It is not hard to check that for all $x> 0$, the optimal control is $u^* = -2$ and for all $x<0$, $u^* = -2$. Therefore, we could easily compute the value function as 
\begin{equation}\label{eqn:VS_example}
        \clvf(x) = \begin{cases}
        \frac{2}{\gamma} e^{ \frac{\gamma |x| -2}{2}}  & |x|>\frac{2}{\gamma}\\
        |x| & |x|\leq \frac{2}{\gamma}
    \end{cases},
\end{equation}
with gradient 
\begin{align*}
    \frac{d\,\clvf}{dx} = \begin{cases}
        sign(x)\cdot e^{\frac{-\gamma |x|-2}{2}} & |x|>\frac{2}{\gamma} \\
        sign(x) & | x| \leq \frac{2}{\gamma}  
    \end{cases}.
\end{align*}
The only non-differentiable point is $x = 0$, with subdifferential $p^- \in [-1,1]$. It can be checked that the value function statsifies~\eqref{eqn:CLVF-VI} in the viscosity sense. However,  
\begin{equation}\label{eqn:counter_example}
    U(x) = \begin{cases}
        e^{\frac{ \gamma (|x| -b )}{2}}  & |x|>a\\
        |x| & |x|\leq a
        \end{cases}.
\end{equation}
is also one viscosity solution for~\eqref{eqn:CLVF-VI}, with any $a\gamma = 2$, $b = \frac{2-2\ln {a}}{\gamma}$. {Also, $U$ also satisfies that $U(x) = 0$ at SRCIS, and $U(x) \rightarrow \infty$ as $x \rightarrow \infty$. This means specifying boundary conditions on the SRCIS or on the boundary of $\dom$ cannot enforce the uniqueness of the viscosity solution.} The R-CLVF~\eqref{eqn:VS_example} is the viscosity solution for any $\gamma$. 

{This means that if we directly solve the R-CLVF-VI, we might not get the R-CLVF as desired. However, this is not a problem both in theory and for the numerical computation. In theory,  implies Proposition~\ref{prop:Vdot<=-gammaV}, which is vital to prove one main result (Theorem~\ref{thrm:CLVF_finite_eq_EC}). For numerical computation, we build up the solver based on the Level-set Toolbox~\cite{ian2005levelset}, which is only applicable to time-varying PDEs (like equation~\eqref{eqn:HJI-VI} and~\eqref{eqn:TV-R-CLVF-VI}). More specifically, we discretize the state space, solve~\eqref{eqn:TV-R-CLVF-VI} and backpropagate using DP until convergence to obtain~\eqref{eqn:infinite_time_CLVF}, instead of directly solving~\eqref{eqn:CLVF-VI}.}



Though uniqueness is not guaranteed, we can still provide the following result from Theorem~\ref{thm:CLVF_VI_viscosity_solution}.
\begin{proposition} \label{prop:Vdot<=-gammaV} 
At any state (differentiable or non-differentiable) in the domain $\dom$ of the R-CLVF, $\forall \dstb \in \dset$, there exists some control $\ctrl \in \cset$ such that  
 \begin{align} \label{eqn:Vdot<=-gammaV}
    \max _{\dstb \in \dset} \min _{\ctrl \in \cset}\dot V^\infty_\gamma \leq -\gamma \clvf.
\end{align} 

\begin{proof}
    See the Appendix~\ref{App: proof_prop4}
\end{proof}

\end{proposition}
 This Proposition is vital for the proof of Lemma~\ref{lemma:Exp_stable}. {Also note that in the proof, the super-differential and sub-differential are used for non-differentiable points.} 

\subsection{R-CLVF implies Exponential Stabilizability} \label{Sec:R-CLVF-EXP}
{As we have already shown that one sufficient condition of the existence of the R-CLVF is the exponential stabilizability in Lemma~\ref{lemma:CLVF_exists}, we will now show that this is also a necessary condition.}



\begin{lemma} \label{lemma:Exp_stable}
    The system can be exponentially stabilized to its smallest robustly control invariant set $\mrcis$ from $\dom \setminus \text{interior}(\mrcis)$ (or $\mathbb R^n  \setminus \text{interior}(\mrcis$)), if the R-CLVF exists in $\dom$ (or $\mathbb R^n$). {Further, $\dom \subset \roes$.}
    
    \input{proofs/proof_ExiimplyExp}
\end{lemma}

This means for a complex nonlinear system, we could find its SRCIS, and check whether (and from where) it can be exponentially stabilized to its SRCIS by computing the R-CLVF of it. More specifically, it finds the maximum region, from where the system can be stabilized to its SRCIS under the user-specified exponential rate $\gamma$, despite worst-case disturbance. 

Combining Lemma~\ref{lemma:CLVF_exists} and Lemma \ref{lemma:Exp_stable}, we directly have the following theorem. 
\begin{theorem} \label{thrm:CLVF_finite_eq_EC} 
The system can be exponentially stabilized to its SRCIS $\mrcis$ from $\dom \setminus \mrcis$ (or $\mathbb R^n  \setminus \mrcis$), if and only if the R-CLVF exists in $\dom$ (or $\mathbb R^n$).  
\end{theorem}

This Theorem extends the classic `CLF' results that stabilize systems to the origin in two ways. First, it is applicable to a broader class of systems (i.e., systems with no equilibrium points). Second, it guarantees the exponential rate, which is specified by the user.  

{\begin{remark} \label{remark:tradeoff}
    From \eqref{eqn:finite_time_CLVF} and \eqref{eqn:infinite_time_CLVF}, it can be seen that if $\gamma _1 > \gamma _2$, then {$V^\infty_{\gamma_1} \geq V^\infty_{\gamma_2}$}. Assume their corresponding domain is $\mathcal D_{\gamma_1}$ and $\mathcal D_{\gamma _2}$, we have $\mathcal D_{\gamma_1} \subseteq \mathcal D_{\gamma_2}$. From Theorem \ref{thrm:CLVF_finite_eq_EC}, we conclude that a larger $\gamma$ corresponds to a faster convergence rate, while a smaller ROES. The user can trade off between a faster convergence rate and a larger ROES. 
\end{remark}}

{There is still one question to be answered: given that the R-CLVF with $\gamma$ exists on $\dom$, is $\gamma$ the fastest exponential rate of convergence? Revisting the example used in Remark.~\ref{remark:counter_ex}, we could see that if there exists a control signal for all possible disturbance strategies s.t. the trajectory can \textbf{reach} the $\mrcis$ in finite time (meaning $\traj \in \mrcis$ at some $\tvar$), the R-CLVF value at that state will be finite for all $\gamma >0$. Further, if this is the case for all states in an open set, then the R-CLVF exists for all $\gamma>0$ on that open set. This is what happened in that 1D example. }

{Further, in practice, we typically do not know a priori whether the system is exponentially stabilizable. Consequently, a line search over $\gamma$ may be needed. In this case, it is not necessary to run the computation over the entire discretized grid. Instead, one can construct a smaller grid containing a subset of the SRCIS and perform the computation on this reduced domain, which can significantly decrease the required computational resources. }

\subsection{R-CLVF-QP}\label{Sec:RCLVF-QP}

For a control and disturbance-affine system,
\begin{equation} \label{eqn:control_affine_system}
    \dot \state = \dyn \big( \state ,\ctrl , \dstb \big) = f_x(\state) + g_u(\state)\ctrl + g_d(\state)\dstb,
\end{equation} 
where $f_x:\mathbb R^n \rightarrow \mathbb R^n$, $g_u:\mathbb R^n \rightarrow \mathbb R^{n\times m_u}$, $g_d:\mathbb R^n \rightarrow \mathbb R^{n\times m_d}$. 
Then, \eqref{eqn:Vdot<=-gammaV} is equivalent to a linear inequality in $u$:
\begin{align*}
     &D_x \clvf(\state) \cdot f_x(\state) + \min _{\ctrl \in \cset} D_x \clvf (\state) \cdot g_u(\state) \ctrl \\
     &\hspace{5em} + \max _{\dstb \in \dset} D_x \clvf (\state) \cdot g_d(\state) \dstb \leq -\gamma V^{\infty}_{\gamma} (\state).
\end{align*} 

\begin{theorem} \label{Thrm: CLVF-QP}
\textbf{(Feasibility Guaranteed R-CLVF-QP)}
{Given $\clvf$ and a reference control $u_r$, the optimal controller can be synthesized by the following CLVF-QP with guaranteed feasibility $\forall x \in \mathcal D_\gamma$ and $0 < \bar \gamma \leq \gamma$.} 
    \begin{align*} 
   &k(x) = u^* = \arg \min _{u\in \mathcal U} \quad (u-u_{r})^T(u-u_{r}), \\ 
    \text{subject to} \hspace{1em}
         &D_x \clvf(\state) \cdot f_x(\state  ) + D_x \clvf (x)\cdot g_u(\state  ) \ctrl \\
         & \hspace{2em} +  \max _{\dstb \in \dset} D_x \clvf (x)\cdot g_d(\state  ) \dstb \leq -\bar \gamma \clvf (\state).
    \end{align*} 
\end{theorem}

\begin{proof}
This is a direct result of Proposition \ref{prop:Vdot<=-gammaV}.
\end{proof}

Note that the QP controller is only point-wise optimal, with respect to ``staying close to the reference controller.'' It is not optimal with respect to the value function. Further, since the R-CLVF is only Lipschitz continuous, its gradient may not be continuous; hence, the QP solution $u = k(x)$ is also not continuous. This may cause the solution of the closed-loop system to lose its uniqueness guarantee~\cite{4518905}. However, such a problem can be solved by considering the sample-and-hold solution as introduced in~\cite{clarke1997asymptotic}. The sample-and-hold solution can be viewed as treating the input of the feedback law as a piecewise-continuous (in $\tvar$) input signal, and therefore, the existence and uniqueness can be guaranteed. Further, this type of solution matches the numerical implementation, {because we usually do not have an analytic feedback controller}. 

{Note that the R-CLVF is computed for a given rate parameter $\gamma$, whereas in the R-CLVF-QP constraint we enforce a (possibly smaller) rate $\bar\gamma$, which relaxes the inequality. With this relaxation, the QP only needs to satisfy a weaker decrease condition, so the optimizer may avoid switching the sign induced by the nonsmooth R-CLVF and might yield a continuous feedback in practice. This makes the QP controller an attractive alternative to the HJ optimal policy, which is typically realized as a bang–bang control.}

In fact, the relation between the stabilizability, the existence of a CLF, and the synthesis of smooth feedback controllers is quite tricky. Even if a continuously differentiable CLF is obtained, we can only guarantee to synthesize a continuous feedback controller, and the resulting closed-loop system will still face the problem of non-existence and non-uniqueness of its solution in the classic sense. Differential inclusion is another popular approach that is used to solve this issue~\cite{clf_zubov}.

    



\section{NUMERICAL BENEFITS}
\label{sec: speedup}

In the numerical computation of the R-CLVF, equation \eqref{eqn:CLVF_DPP} is solved on a discrete grid, until some convergence threshold is met, which leads to the well-known ``curse of dimensionality.'' In this section, we provide two main methods to overcome this issue: the warmstarting technique and the system decomposition technique. Necessary proofs are provided, and the effectiveness is validated with a 10D example in the numerical example.

\subsection{R-CLVF with Warmstarting}
In the previous work, we introduced a two-step process: first, the SRCIS is computed, then the R-CLVF is computed. This process requires solving the TV-R-CLVF-VI two times, each with different initializations. In this subsection, we show that the converged value function for the first step can be used to warmstart the second step computation. 

Denote the TV-R-CLVF with initial value $k(\state)$ as $\bar {V}_\gamma (\state,\tinit)$, and the infinite time value function as $\bar {V}^\infty_\gamma (\state)$, with the corresponding domain $\bar {\mathcal D}_\gamma$. We only change the initial value, and still have the same loss function $\loss(\state)$ for $\bar {V}_\gamma (\state,\tinit)$ and $\tclvf(\state,\tinit)$.

\begin{theorem}\label{thrm:R-CLVF_warmstarting_inexact}
    For all initialization $\bar {V}_\gamma(\state,\thor) = k(x)$, we have $\bar {V}_\gamma(\state,\tinit) \geq  {V_\gamma}(\state,\tinit) $ holds $\forall \state$, $\forall \tinit < \thor$.
    
    \input{proofs/proof_thrm_IEwarmstart}
\end{theorem}

Theorem~\ref{thrm:R-CLVF_warmstarting_inexact} shows that no matter what the initial condition is, the value function propagated with this initial condition is always an over-approximation of the TV-R-CLVF. This also holds for the R-CLVF.

\begin{proposition} \label{prop: Vk>V}
    If $\bar {V}_\gamma ^\infty(\state)$ exists on $\bar {\mathcal D_\gamma }$, then $\bar {V}_\gamma ^\infty(\state) \geq  \clvf(\state)$ and $\bar {\mathcal D}_\gamma  \subseteq  \mathcal D_\gamma $.
    
    \begin{proof}
        The first part is a direct result from Theorem \ref{thrm:R-CLVF_warmstarting_inexact}. The second part can be proved by contradiction. Assume $\state\ \in \bar {\mathcal D}_\gamma $ but $ \state \notin  {\mathcal D_\gamma }$. This means $\bar {V}_\gamma ^\infty(\state)$ is finite, but $ {V_\gamma ^\infty}(\state)$ is infinite, which contradicts the first part of this proposition. 
    \end{proof}
\end{proposition}

{The above results concern inexact warm-starting, which cannot be used directly in most cases, since our goal is to recover the exact R-CLVF. Nevertheless, these results are essential for establishing the exact warm-starting presented below. In particular, we show that with appropriate initializations, the warm-started procedure can recover the exact R-CLVF.
}
\begin{theorem}\label{thrm:R-CLVF_warmstarting_exact}
    For initialization $\bar {V}_\gamma (\state,\thor) = k(\state) \leq \clvf(\state)$, we have $ \bar {V} _\gamma ^\infty(\state) =  \clvf(\state)$.
    
\input{proofs/proof_thrm_Ewarmstart}
\end{theorem}
\input{algo_compute}

Using Theorem \ref{thrm:R-CLVF_warmstarting_exact}, we provide an enhanced version of the original algorithm for computing the R-CLVF, shown in Alg.~\ref{Algorithm:CLVF}. The main difference is that after finding the SRCIS and the corresponding value function $V^\infty (\state)$ (line 5), the next step computation (line 9) is initialized with $V^\infty (\state ) - \minval$, instead of $\ell (\state)$. The exact warmstarting is guaranteed, because we can always guarantee $ V^\infty (\state) - \minval \leq \clvf(\state)$. From the numerical examples, Alg.~\ref{Algorithm:CLVF} accelerates the computation from 5\% to 90\%.


\subsection{R-CLVF with Decomposition}

{We first introduce the self-contained subsystems decomposition. }

\begin{definition}
\label{def: SCSD}
(Self-contained subsystem decomposition) (SCSD) Given system \eqref{eqn:dynamic_system} and assume there exists state partitions $\subs_1 = (\state_1,\state_c) \in \tsset_1$, $\subs_2 = (\state_2,\state_c) \in \tsset_2$, where $\state_1\in\mathbb{R}^{n_1}$, $\state_2\in\mathbb{R}^{n_2}$, $\state_c\in\mathbb{R}^{n_c}$, $n_1,n_2>0$, $n_c\ge 0$, $n_1+n_2+n_c=n$. Assume also the control and disturbance inputs can be partitioned similarly with $v_1 = (u_1,u_c) \in \mathcal V_1$, $v_2 = (u_2,u_c) \in \mathcal V_2$, where $u_1 \in \mathbb R^{m_1}$, $u_2 \in \mathbb R^{m_2}$, $u_c \in \mathbb R^{m_c}$ and $m_1+m_2+m_c = m$. $p_1 = (d_1,d_c) \in \mathcal P_1$, $p_2 = (d_2,d_c) \in \mathcal P_2$, where $d_1 \in \mathbb R^{p_1}$, $d_2 \in \mathbb R^{p_2}$, $d_c \in \mathbb R^{p_c}$ and $p_1+p_2+p_c = p$. Given the system \eqref{eqn:dynamic_system}, the two subsystems of it  are 
\begin{align*} 
    \dot z_1 = f_1(z_1,v_1,p_1), \quad
    \dot z_2 = f_2(z_2,v_2,p_2).
\end{align*}
Here, $x_c$, $u_c$, $d_c$ are called the shared state, control, and disturbances respectively.
\end{definition}



\begin{theorem}\label{thrm:R-CLVF_decomposition}
    Assume the system can be decomposed into several self-contained subsystems, and there are no shared control and states between each subsystem. Denote the corresponding R-CLVFs for the subsystems as $V_{\gamma,i}^\infty (z_i)$ with domain $\mathcal D_{\gamma_i,z_i}$ and loss $\loss_i$, and define 
    \begin{align} \label{eqn: reconstructCLF}
        W_\gamma^\infty(\state) = \max_{i} V_{\gamma,i}^\infty (z_i).
    \end{align}
    Then, $\loss(\state) = \max_{i} \loss(\subs_i)$ implies $W_\gamma^\infty(\state)$ is the R-CLVF of system~\eqref{eqn:dynamic_system}.
    
    \begin{proof}
        {The proof can be obtained analogously following Lemma 1 of~\cite{gong2024synthesizing}, and we provide a sketch here. We first prove that $\loss(\state) = \max_{i} \loss(\subs_i)$ implies 
        \begin{align*}
             W_\gamma(\state,\tinit) = \max_{i} V_{\gamma,i} (z_i,\tinit)
        \end{align*}
        by contradiction. More specifically, we assume $W_\gamma < V_\gamma$, and show that the optimal controls of $V_{\gamma,i}$ can be used to construct a control for the original system. We then use the definition of $W_\gamma$ to show that the reconstructed control results a value lower than the optimal control for $V_\gamma$, which is a contradiction. The opposite direction can be proved with same procedure. }
    \end{proof}
\end{theorem}

\section{NUMERICAL EXAMPLES}
In this section, we provide three examples to showcase the main benefits of using R-CLVF: 1) it handles general nonlinear dynamics with bounded control and disturbance, 2) it finds and stabilizes the system to its SRCIS (based on different norms chosen) given a user-specified exponential rate $\gamma$, 3) with warmstarting and decomposition, the computational cost is decreased significantly. All examples are solved using MATLAB and toolboxes \cite{mitchell2007toolbox,chenoptimal}. All trajectories are generated with QP controller~\eqref{Thrm: CLVF-QP} with reference control $u_r = 0$.

{Since the numerical value function is obtained on a discrete grid, we use interpolation methods to obtain the value for the states that are not grid points. The R-CLVF-VI may not hold for those states due to the interpolation, and the R-CLVF-QP may lose the feasibility guarantee. To solve this problem, before solving the R-CLVF-QP, we check its feasibility and find the minimal slack variable to make it feasible by solving another optimization problem.}

\label{sec: examples}
\subsection{2D System Revisit}
Consider again the system~\eqref{eqn:EX_IN2D}, and specify $ \hjloss(\state) = ||x||_\infty$. We compute the R-CLVF with $\gamma_1 = 0.1$, $\gamma_2 = 0.3$. The results are shown in Fig.~\ref{fig:IN2D_traj}. It should be noted that for this system, the SRCIS for $\gamma = 0.1$ and $\gamma = 0.2$ are both $\mrcis = \{|x| \leq 0.5, |y| \leq 0.5\}$, and ROES $\roes = \{|x| > 0.5, |y| < 1\} \setminus \mrcis$. 

\begin{figure}[t]
\centering
\includegraphics[width=\columnwidth]{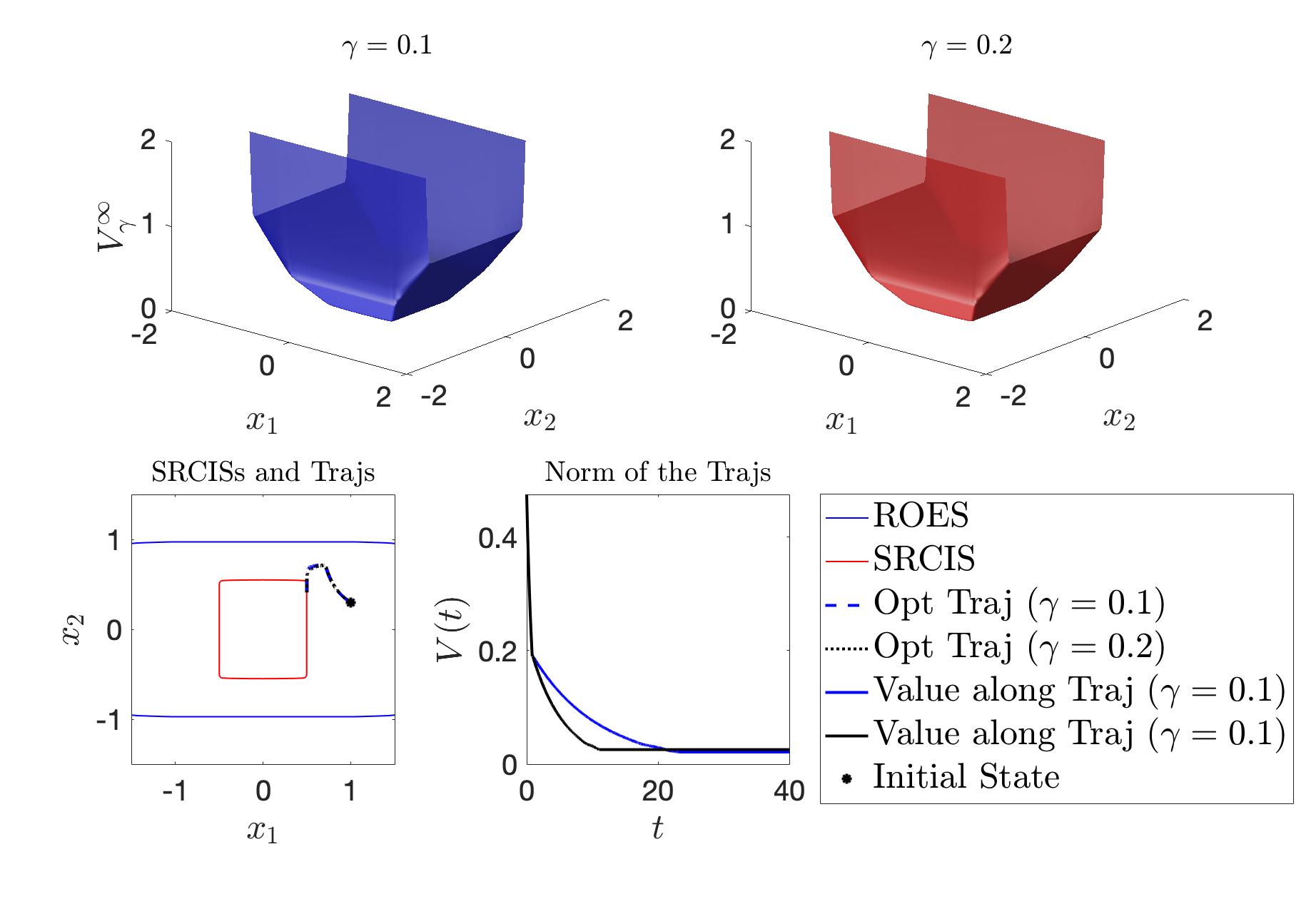}
    \caption{Top: R-CLVF with $\gamma = 0.1$ (left) and $\gamma = 0.2$ (right). Bottom left: ROES, SRCIS, and the two optimal trajectories using R-CLVF-QP controller. The ROES and SRCIS for different $\gamma$ are all the same, while the optimal trajectories are different. To see this, first consider a point on the boundary of ROES, $[0.1,1]$, $d$ will make $x$ increase 0.1 to 1, while $u$ cannot decrease $y$. Since the distance is measured by $||x||_\infty$, we have $\loss (\opttraj) = 1$, $\forall t < 0$. Using equation~\eqref{eqn:finite_time_CLVF} and~\eqref{eqn:infinite_time_CLVF}, the value will be infinite. However, for any $|y|<1$, the control can decrease $y$ to $0$, and for all $x$, it either goes to 0.5 or -0.5. Note both happen in a finite time horizon. Therefore, using equation~\eqref{eqn:finite_time_CLVF} and~\eqref{eqn:infinite_time_CLVF}, the value will be finite for all $\gamma \geq 0$. Bottom mid: value decay along the two optimal trajectories. All controllers were generated using R-CLVF-QP. With a 151-by-151 grid, the computation time for $\gamma = 0.1$ is 215.6s with warmstarting, and 289.7s w/o warmstarting, and 211.5s with warmstarting, and 258.4s w/o warmstarting for $\gamma=0.2$.
    }
    \label{fig:IN2D_traj}
\end{figure}

\subsection{3D Dubins Car}
Consider the 3D Dubins car example:
\begin{align*}
    \dot x = v \cos (\theta) + d_x, \hspace{1em} \dot y = v \sin(\theta) +d_y, \hspace{1em} \dot \theta = u,
\end{align*}
where $v = 1$ and $u \in [-\pi/2 , \pi/2] $ is the control and $d_x,d_y \in [-0.1,0.1]$ is the disturbance. This system has no equilibrium point. The SRCISs with different $ \hjloss(\state)$ are shown in Fig.~\ref{fig:Dubins}, and the trajectory converges to the SRCIS exponentially.

\begin{figure}[t]
\centering
\includegraphics[width=\columnwidth]{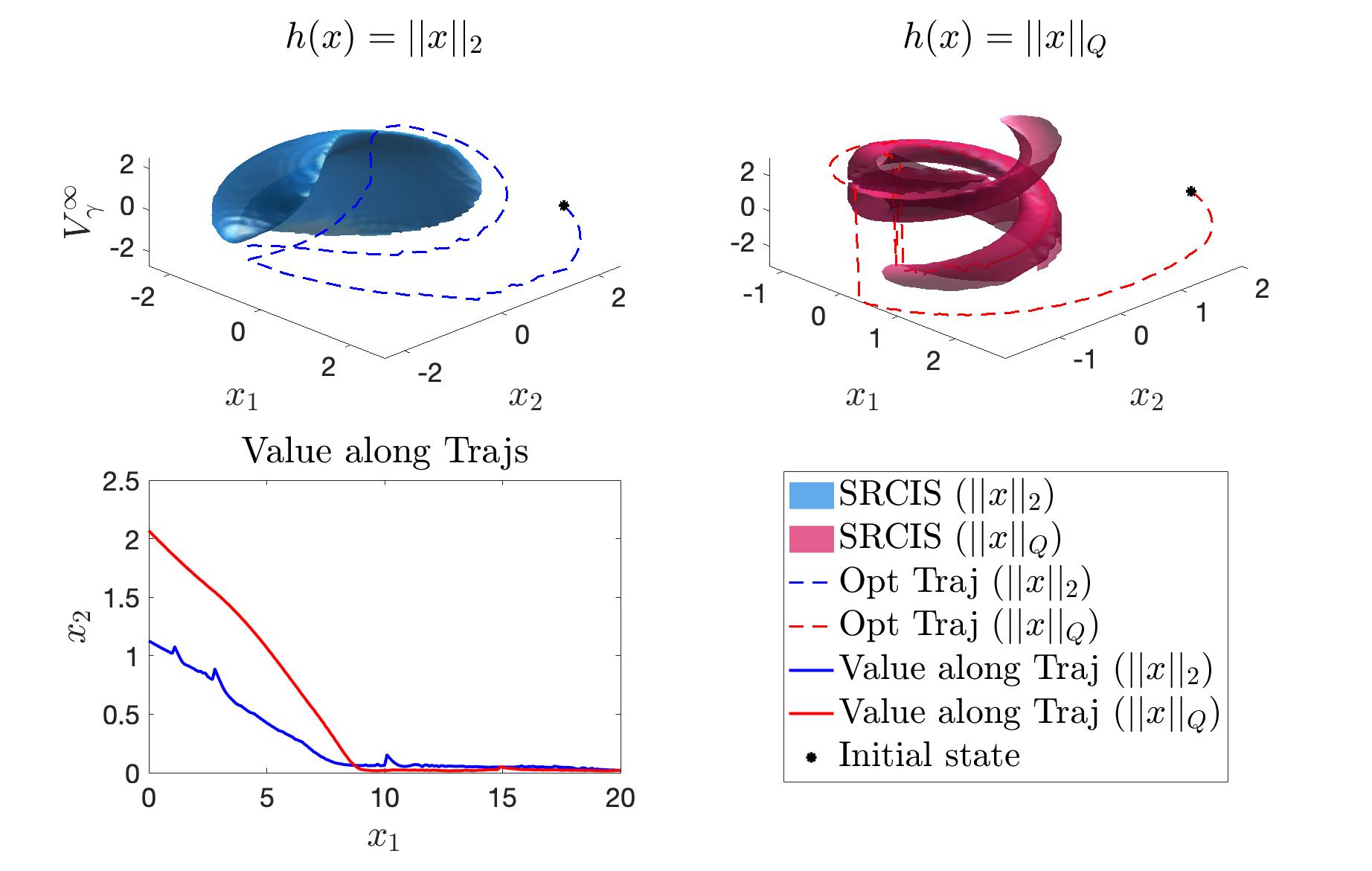}
    \caption{Different SRCISs with different $ \hjloss (\state)$. Top left: SRCIS and optimal trajectory with $\hjloss  (\state) = ||\state||_2$. Top right: SRCIS and optimal trajectory with $\hjloss (\state) = ||\state||_Q$, where $Q = diag[1,1,0]$. Bottom left: the value along the optimal trajectories. All controllers were generated using R-CLVF-QP. With a 51-by-51-by-53 grid, the computation time for $ \hjloss(\state) = ||\state||_2$ is 264s with warmstarting, and 386.6s w/o warmstarting, and 143.4s with warmstarting, and 207.7s w/o warmstarting for $ \hjloss(\state) = ||\state||_Q$.
    }
    \label{fig:Dubins}
\end{figure}

\begin{figure}[t]
\centering
\includegraphics[width=\columnwidth]{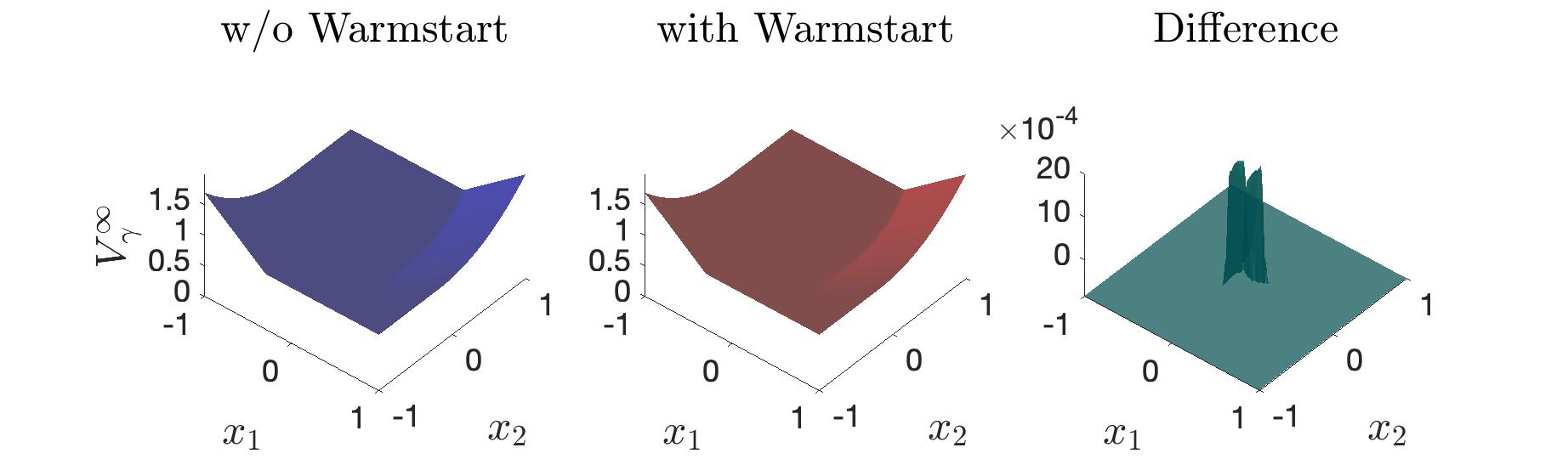}
    \caption{Comparison of R-CLVF with and without warmstarting for the Z-subsystem. The difference is almost negligible.
    }
    \label{fig:Zdim}
\end{figure}

\subsection{10D Quadrotor}
Consider the 10D quadrotor system: 
\begin{align}
    &\dot{x} = v_x + d_x, \hspace{2mm} \dot{v_x} =  g\tan{\theta_x}, \hspace{2mm} \dot{\theta_x} = -d_1\theta_x + \omega_x, \notag \\ 
    &\dot{\omega_x} =  -d_0\theta_x + n_0u_x, \hspace{2mm} \dot{y} =  v_y + d_y, \hspace{2mm}  \dot{v_y} = g\tan{\theta_y}, \notag \\
    & \dot{\theta_y} = -d_1\theta_y + \omega_y, \hspace{2mm} \dot{\omega_y} = -d_0\theta_y + n_0u_y, \notag \\
    \
    & \dot{z} = v_z + d_z, \hspace{2mm} \dot{v_z} = u_z, \label{eq: 10D_Quad}
  \end{align}
where $(x,y,z)$ denote the position, $(v_x, v_y, v_z)$ denote the velocity, $(\theta_x, \theta_y)$ denote the pitch and roll, $(\omega_x, \omega_y)$ denote the pitch and roll rates, and $(u_x, u_y, u_z)$ are the controls. The parameters are set to be $d_0 = 10, d_1 = 8, n_0 = 10, k_T = 0.91, g = 9.81$, $|u_x|, |u_y| \leq \pi/9$, $u_z\in [-1, 1]$, $|d_x|, |d_y|, |d_z| \leq 0.1$.

This 10D system can be decomposed into three subsystems: X-sys with states $[x,v_x,\theta_x,\omega_x]$,  Y-sys with states $[y,v_y,\theta_y,\omega_y]$, and Z-sys with states $[z,v_z]$. It can be verified that all three subsystems have an equilibrium point at the origin. Further, there's no shared control, disturbance, or state among subsystems. We use $ \hjloss(\state) = ||\state||_\infty$, which satisfies the condition $\loss(\state) = \max_i \loss_i(z_i)$. The R-CLVF is reconstructed using equation~\eqref{eqn: reconstructCLF}. 

A comparison of the R-CLVF for the Z-sys with and without warmstarting is shown in Fig.~\ref{fig:Zdim}, showing that the warmstarting provides the exact result. The trajectory is shown in Fig.~\ref{fig:10D}.

\begin{figure}[t]
\centering
\includegraphics[width=\columnwidth]{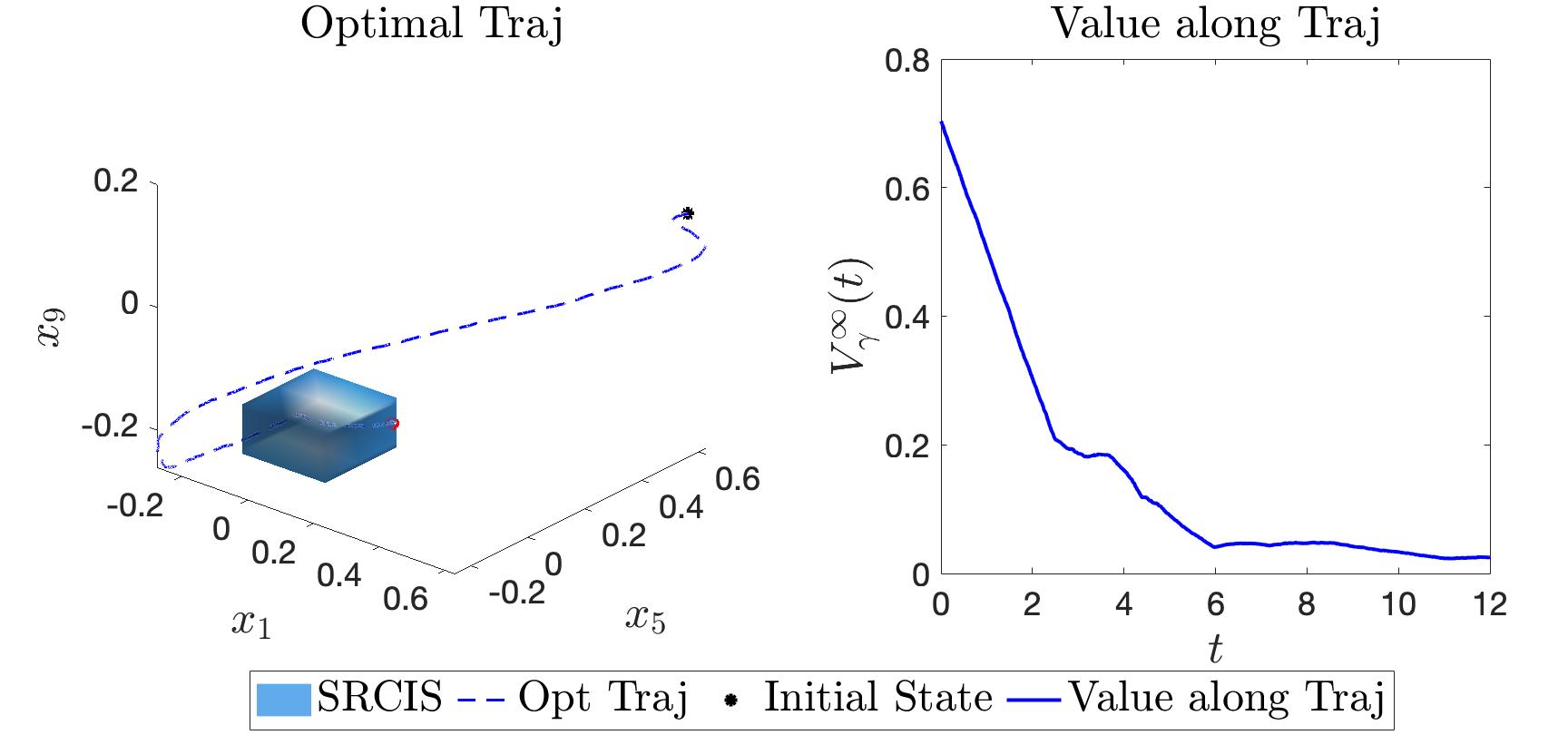}
    \caption{Left: SRCIS of the reconstructed R-CLVF and the optimal trajectory. For Z-sys, with 101 grids for each state, time step = 0.1, convergence threshold = 0.0015, the computation time is 36.59s with warmstarting and 42.72s w/o warmstarting. For X(Y)-sys, with 17 grids for each state, time step = 0.1, convergence threshold = 0.02, the computation time is 828.27s with warmstarting and 887.79 w/o warmstarting. Right: the decay of the value along the optimal trajectory. 
    }
    \label{fig:10D}
\end{figure}




\section{Conclusions}

In this paper, we extended our preliminary work on constructing CLVFs using HJ reachability analysis to the system with bounded control and disturbances and to stabilize to a random point of interest. We provided more detailed discussions on several important Lemmas and Theorems. Additionally, warmstarting and decomposition methods are proposed to overcome the ``curse of dimensionality,'' and the effectiveness of both techniques is validated with numerical examples. 
Future directions include finding conditions under which the self-contained subsystem decomposition method provides the R-CLVF, and incorporating learning-based methods to tune the exponential rate $\gamma$ for online execution in robotics applications. 


    
\begin{appendix}
\section{Appendix}

\subsection{Proof of Proposition~\ref{prop: TV-R-CLVF_lipschitz}} \label{App: proof_prop1}
\input{proofs/proof_prop_TV_lipschitz}

\subsection{Proof of Theorem~\ref{thm:TV-R-CLVF-VI-VS}} \label{App: proof_thrm2}

\input{proofs/proof_thrm_TV_VS_arxiv}

\subsection{Proof of Proposition~\ref{Prop:R-CLVF-lipschitz}}\label{App: proof_prop2}
\input{proofs/proof_thrm_CLVF_Lipschitz}

\subsection{Proof of Proposition~\ref{prop: Radially_unbounded}}\label{App: proof_prop3}
\input{proofs/proof_prop_radiallyUnbounded}

\subsection{Proof of Proposition~\ref{prop:Vdot<=-gammaV}}\label{App: proof_prop4}
\input{proofs/proof_VdotleqV}



\end{appendix}

\bibliographystyle{IEEEtran}
\bibliography{ref}

\end{document}

%% file: package.tex
\usepackage{graphicx}
\usepackage{booktabs}
\usepackage{amsmath} 

\DeclareMathOperator*{\argmax}{argmax}
\let\labelindent\relax
\usepackage{enumitem}
\newcommand{\subscript}[2]{$#1 _ #2$}
\usepackage{amsfonts}
\usepackage{amssymb}  
\usepackage{mathtools} 
\usepackage{mathrsfs}
\usepackage{balance}
\usepackage{url}
\usepackage{hyperref}
\usepackage{algorithm}
\usepackage{algpseudocode}
\DeclareMathAlphabet{\mathpzc}{OT1}{pzc}{m}{it}
\usepackage{makecell}
\usepackage[noadjust]{cite}
\usepackage{xcolor}
\usepackage{comment}


%% file: notation.tex
\newcommand{\shnote}[1]%
    {\textcolor{magenta}{ #1}}

\newcommand{\zgnote}[1]%
    { \textbf{\textcolor{red}{#1}} }
    
\newcommand{\zgrvs}[1]%
    { \textbf{\textcolor{red}{#1}} }

\newcommand{\zgmv}[1]%
    { \textcolor{blue}{#1}}

\newcommand{\dst}{\text{dst}} 
\newcommand{\subs}{z} 
\newcommand{\tsset}{\mathcal{Z}}
\newcommand{\tinit}{t}
\newcommand{\tvar}{s}
\newcommand{\thor}{0} 
\newcommand{\state}{x} 

\newcommand{\dyn}{f} 
\newcommand{\poi}{p} 

\newcommand{\ctrl}{u} 
\newcommand{\dstb}{d} 
\newcommand{\csig}{\ctrl(\cdot)} 
\newcommand{\dsig}{\dstb(\cdot)} 
\newcommand{\csigopt}{\ctrl^*(\cdot)} 
\newcommand{\dsigopt}{\dstb^*(\cdot)}%
\newcommand{\dmapopt}{\dmap^*[\ctrl]} 
\newcommand{\dset}{\mathcal{D}}
\newcommand{\dfset}{\mathbb{D}}
\newcommand{\dmap}{\lambda} 
\newcommand{\Dmap}{\Lambda} 

\newcommand{\hjloss}{h}
\newcommand{\loss}{\ell}

\newcommand{\cset}{\mathcal{U}} 
\newcommand{\cfset}{\mathbb{U}} 
\newcommand{\traj}{\xi(\tvar;\tinit, \state,\csig,\dsig)} %
\newcommand{\trajmap}{\xi(\tvar;\tinit, \state,\csig,\dmap)}

\newcommand{\opttraj}{\xi(\tvar;\tinit, \state, u^*(\cdot),\dmap[u^*])}

\newcommand{\rcis}{\mathcal{I}}
\newcommand{\mrcis}{\mathcal{I}_m}
\newcommand{\roes}{\mathcal{D}_\text{ROES}}
\newcommand{\dom}{\mathcal{D}_\gamma}

\newcommand{\valfunc}{V} 
\newcommand{\tclvf}{\valfunc_\gamma} 
\newcommand{\clvf}{\valfunc_\gamma^\infty} 

\newcommand{\minval}{V^\infty_m} 

\newtheorem{remark}{Remark}
\newtheorem{proposition}{Proposition}
\newtheorem{definition}{Definition}

\newtheorem{theorem}{Theorem}
\newtheorem{lemma}[theorem]{Lemma}

%% file: Introduction.tex
\textit{Liveness} and \textit{safety} are two main concerns for autonomous systems working in the real world. Using control Lyapunov functions (CLFs) to stabilize the trajectories of a system to an equilibrium point~\cite{sontag1989universal,freeman1996control,hassan2002nonlinear} is a popular approach to ensure liveness, whereas using control barrier functions (CBFs) to guarantee forward control invariance is popular for maintaining safety~\cite{cbfsurvey, ames2014rapidly,cbfqptac}. However, identifying valid CLFs and CBFs is challenging for general nonlinear systems, and users of these methods typically rely on hand-designed or application-specific CLFs and CBFs \cite{artstein1983stabilization,clf_zubov,giesl2015review,giesl2008construction,xu2015robustness}. Finding these hand-crafted functions can be difficult, especially for high-dimensional systems with state and/or input constraints.


Liveness and safety can also be achieved by formal methods such as Hamilton-Jacobi (HJ) reachability analysis~\cite{hjreachabilityoverview}. This method formulates \textit{liveness} and \textit{safety} as optimal control problems, and has been used for applications in aerospace, autonomous driving, and more~\cite{evans_hj,bardi2008optimal,crandall1983viscosity,crandall1984some,frankowska1989hamilton}. This method computes a value function whose level sets provide information about safety (or liveness) over space and time, and whose gradients inform the safety (or liveness) controller. This value function is the unique viscosity solution to a Hamilton-Jacobi-Issac's Variational Inequality (HJI-VI), can be computed numerically using dynamic programming (DP) for general nonlinear systems, and can accommodate input and disturbance bounds. Undermining these appealing benefits is the curse of dimensionality: computation scales exponentially with state dimension. Ongoing research has improved computational efficiency~\cite{decomposition,bansal2021deepreach,herbert2021safelearning,10365682}, but performing DP in high dimensions (6D or more) remains challenging.

Standard HJ reachability analysis focuses on problems such as the minimum time to reach a goal, or how to avoid certain states for a finite (or infinite) time horizon. It does not stabilize a system to a goal after reaching it. In our previous work~\cite{gong2022constructing}, we modified the value function and defined the control Lyapunov value function (CLVF) for undisturbed nonlinear systems with compact control. The CLVF finds the smallest control invariant set and the region of exponential stabilizability (ROES) of the system given a user-specified exponential rate $\gamma$. {The original work is restrictive and suffers from the curse of dimensionality. }

Other works that study the relationship between CLF and some partial differential equations (PDE) include Zubov's method and its extensions~\cite{zubov1961methods,grune2000computing,dubljevic2002new,clf_zubov}.{~\cite{zubov1961methods} is the first that relates the Lyapunov function with a PDE, called Zubov's equation. Later,~\cite{grune2000computing,dubljevic2002new,clf_zubov} extends this early work to systems with control. The general process is to define an optimal control problem, whose value function itself is a CLF and is the viscosity solution to some PDE of Zubov's type. Therefore, the CLF (which is the value function) can be found by solving the PDE. Viscosity solutions are considered because these PDEs usually do not admit smooth solutions.} Another work that considers the relation between the state-constrained optimal control problem and the HJI-VI is~\cite{altarovici2013general}, where an auxiliary unconstrained optimal control problem is solved using the HJI-VI. However, this requires augmenting to one additional dimension, which requires much larger computational resources. {Further, there also exist works on computing the region of attraction of a given system. These methods use optimization like linear matrix inequality~\cite{chesi2009estimating}, sum of squares~\cite {topcu2007stability,topcu2009robust}, or convex program~\cite{henrion2013convex}. However, they usually only guarantee under approximation and require special forms of the dynamics (e.g., polynomial dynamics). In contrast, we provide exact recovery of the ROES for general nonlinear systems, and we consider the convergence to a set, instead of just the equilibrium point.} 

This work seeks to {extend our previous CLVF work, correct the errors, and provided numerical feasible approach for high-dimensional systems.} The main contributions are: 
\begin{enumerate}
    \item {We define the smallest robustly control invariant set (SRCIS) of a point of interest (POI) given system dynamics. We extend our previous work~\cite{gong2022constructing} and define the robust-CLVF (R-CLVF), which is with respect to an arbitrary POI, instead of an equilibrium point. }
    \item {We prove that the R-CLVF is Lipschitz continuous, satisfies the DP principle (DPP), is a viscosity solution to the corresponding R-CLVF variational inequality (VI), and we update the algorithm for computing it. We also show that the SRCIS is the zero sub-level set of the R-CLVF, and that the R-CLVF can stabilize the system to the SRCIS of the POI from the ROES with an exponential rate $\gamma$ defined by the user. We also show that the domain of the R-CLVF is exactly the ROES.}
    \item {Our method uses $\gamma$ as an exponential amplifier, not as a discount factor (which is common in the literature). This amplifier provides straightforward physical intuition for our formulation. }
    \item Two methods to accelerate computation are introduced: warm-starting and system decomposition. We prove that under certain assumptions, these acceleration methods will recover exact R-CLVF. 
\end{enumerate}

The paper is organized in the following order: Section \ref{sec: background} provides background information on HJ reachability analysis. Section \ref{sec: r-clvf} introduces the R-CLVF {and presents the main results of the paper, e.g., the existence of the R-CLVF is equivalent to the exponential stabilizability of the system.} A feasibility-guaranteed quadratic program (QP) controller is provided. Section \ref{sec: speedup} introduces warm-starting and system decomposition to accelerate the computation. Section \ref{sec: examples} shows three numerical examples, validating the theory. 

%% file: Background.tex
In this paper, we seek to exponentially stabilize a given nonlinear time-invariant dynamic system with {compact} control and disturbance to its SRCIS. We start with the background information. 

\subsection{Problem Formulation}\label{sec:Problem Formulation}
Consider the nonlinear time-invariant system  
 \begin{equation} \label{eqn:dynamic_system}
    \dot{\state}(\tvar) = \dyn \left( \state (\tvar ), \ctrl(\tvar) ,\dstb(\tvar) \right), \hspace{0.5em} \tvar \in [\tinit, \thor], \hspace{0.5em} \state(\tinit)=\state_0,
\end{equation} 
where $ \tinit <0$ is the initial time, and $\state_0\in \mathbb R^n$ is the initial state. The control signal $\csig$ and disturbance signal $\dsig$ are drawn from the set of measurable functions $\cfset_t$ and $\dfset_t$. Assume also the control input $\ctrl$ and disturbance $\dstb$ are drawn from convex compact sets $ \cset \subset \mathbb R^m$ and $ \dset \subset \mathbb R^p$ respectively. We have: 
\begin{align*}
     \cfset_t := \{ \csig : [\tinit, \thor] \mapsto \cset, \csig \text{ is measurable} \}, \\ 
     \dfset_t := \{ \dsig : [\tinit, \thor] \mapsto \dset, \dsig \text{ is measurable} \}.
\end{align*}
\input{assumptions}

Under these assumptions, given an initial state $\state$ and control and disturbance signals $\csig$, $\dsig$, there exists a unique solution $\traj$, $\tvar \in  [\tinit,\thor]$ of the system~\eqref{eqn:dynamic_system}. We call this solution the trajectory of the system, and where possible, we use $\xi(\tvar)$ for conciseness. Further assume the disturbance signal can be determined as a strategy with respect to the control signal: $\dmap: \cfset_t \mapsto \dfset_t$, drawn from the set of non-anticipative maps $\dmap \in \Dmap_\tinit $ ~\cite{varaiya1967existence}, defined as:
\begin{align*}
    \Dmap_\tinit := \{ & \mathcal N: A(t) \mapsto B(t): a(\tvar) = \hat a(\tvar) \text{ a.e. } \forall \tvar \in [\tinit,\thor] \\
    & \implies \mathcal N[a](\tvar) = \mathcal N[\hat a](\tvar) \text{ a.e. } \forall \tvar \in [\tinit,\thor] \} .
\end{align*}

In this paper, we seek to stabilize the system~\eqref{eqn:dynamic_system} to its SRCIS. We first introduce the notion of a robust control invariant set (RCIS).

\begin{definition} (RCIS)
    A \textbf{closed} set $\rcis$ is robustly control invariant for \eqref{eqn:dynamic_system} if $\forall \state \in \rcis$, $\forall \dmap \in \Dmap_\tinit$, $\exists \csig \in \cfset_t $ such that $\xi(\tvar; \tinit,\state, \csig, \dmap [\ctrl]) \in \rcis$, $\forall \tvar \in [\tinit,\thor]$.
\end{definition}
We also assume the following:
\input{assumptions2}

Assumption $A_3$ is standard. For $A_4$, some systems might not posses an equilibrium point, e.g., the 3D Dubins car with constant velocity or any systems that do not satisfy the small control property.

\begin{remark}
    $A_4$ is a weaker assumption compared to assuming the existence of a CLF. For a valid CLF, the region of null controllability (RONC) is also robust control invariant. Further, all trajectories starting from the RONC can be stabilized to the equilibrium point. With $A_4$, we extend the classical definition of stabilizing to the equilibrium point to stabilizing to the SRCIS, if the former is impossible to achieve. 
\end{remark}

We are also interested in finding the ROES of a set. We first define the signed distance from a point to a set $\mathcal A$ to be 
\begin{equation} \label{eqn:dst_to_set}
    \dst(\state;\mathcal A) = \begin{cases}
        \min _ {a \in \partial \mathcal A} ||x-a|| & \state \notin \mathcal A ,\\
        -\min _ {a \in \partial \mathcal A} ||x-a|| & \state\in \text{int} (\mathcal A) ,\\
        0 & \state \in \partial \mathcal A.
    \end{cases} 
\end{equation}
where $\partial \mathcal A$ is the boundary of $ \mathcal A$ and any vector norm is applicable here. 
\begin{definition} \label{def:ROES}
The ROES of a set $\rcis$ is the set of states from which the trajectory converges to $\rcis$ with an exponential rate $\gamma$:
\begin{align*}
    &\roes  :=\{ x\in \mathbb R^n | \hspace{2mm} \forall \dmap \in \Dmap_\tinit, \exists \csig \in \cfset_t, \gamma,c > 0
    \text{ s.t. }  \\
    & \dst(\xi(\tvar; \tinit,\state, \csig, \dmap [\ctrl]);\rcis) \leq c e^{-\gamma (\tvar-\tinit)} \dst(\state;\rcis) \}.
\end{align*}
\end{definition}

\subsection{HJ Reachability and SRCIS} \label{sec:Optimal_control_HJB}
In the previous work~\cite{gong2022constructing}, we proposed constructing the CLVF using HJ reachability analysis. This is done by formulating a reachability \textit{safety} problem, where the system tries to avoid all regions of the state space that are not the origin. This problem can be solved as an optimal control problem. 


Traditionally in HJ reachability analysis, a continuous loss function $ \hjloss : \mathbb R^n \mapsto \mathbb R$ is defined such that its zero super-level set is the failure set $\mathcal F = \{\state: \hjloss (\state) > 0 \}$. 
The finite-time horizon cost function captures whether a trajectory enters $\mathcal{F}$ at any time in $[\tinit,\thor]$ under given control and disturbance signals by computing the maximum loss accrued over time: 
\begin{equation} \label{eqn:finite_time_costfunctional}
    J(\tinit, \state, \csig , \dsig ) = \max_{\tvar \in [\tinit, \thor]} \hjloss \bigl(\traj \bigl).
\end{equation} 
The value function is the cost under the optimal control signal and worst-case disturbance: 
\begin{align}\label{eqn:finite_time_value_function}
    V(\state, \tinit) &=\sup _{\dmap \in \Dmap_\tinit} \inf _{ \csig \in \cfset_t} J(\tinit, \state, \csig ,\dmap[\ctrl]), \notag \\
    &=\sup _{\dmap \in \Dmap_\tinit} \inf _{\csig \in \cfset_t} \max_{\tvar \in [\tinit,\thor]} \hjloss (\xi(\tvar;\tinit, \state, \csig, \dmap[\ctrl]) .
\end{align} 
For a given $\tinit$, the (strict) zero super-level set  $\mathcal V_0  = \{\state: V(\state, \tinit) > 0 \}$ denotes the set of initial states such that there exists a disturbance signal that drives the trajectory to $\mathcal F$ for some time $s\in[t,0]$, despite the control signal used. The zero sub-level set of $V(\state, \tinit)$ is therefore \textit{safe} for the time horizon $[t,0]$. This can be extended to say that each $\alpha$ sub-level set $\mathcal V_\alpha  = \{\state: V(\state, \tinit) \leq \alpha \}$ is safe w.r.t. the set defined by $ \mathcal F_\alpha = \{ \state: \hjloss (\state) > \alpha \}$.

The infinite-time horizon value function is defined by taking the limit (if it exists) of $V(\state, \tinit)$ as $t\rightarrow-\infty$ \cite{fialho1999worst}, 
\begin{align}\label{eqn:infinite_time_value_function}
    V^\infty(\state) = \lim_{ \tinit \rightarrow-\infty} V(\state, \tinit).
\end{align} %
Different from the time-varying value function~\eqref{eqn:finite_time_value_function}, for all states in the $\alpha$ sub-level set of $V^\infty(\state)$, there always exists a control signal such that the maximum loss is lower than or equal to $\alpha$ despite the disturbance signal. This means every $\alpha$ sub-level set of $V^\infty(\state)$ is robustly control invariant, and the trajectories can be maintained within a particular level set boundary. Further, this set is the \textit{largest} RCIS contained within the $\alpha$ sub-level set of $\hjloss(\state)$~\cite{choi2021robust}.

It has been shown that~\eqref{eqn:finite_time_value_function} is the unique viscosity solution to the following HJI-VI~\cite{fisac2015reach}:  
 \begin{equation} \label{eqn:HJI-VI}
    \begin{aligned}
         &0= \min \biggl\{\hjloss (\state) - V(\state, \tinit), \\
        & \hspace{1em} D_{\tinit}  V(\state, \tinit) + \max_{ \dstb \in \dset} \min _{\ctrl \in \cset}  D_{\state} V(\state, \tinit) \cdot \dyn(\state, \ctrl , \dstb)  \biggl\},
    \end{aligned}
\end{equation} 
where $D_t$ and $D_x$ denote the derivative over time and state respectively. The value function \eqref{eqn:finite_time_value_function} can be computed numerically using DP by solving this HJI-VI recursively over time. And the infinite-time value function~\eqref{eqn:infinite_time_value_function} can be obtained by solving this HJI-VI until convergence. 


\begin{remark}\label{remark:physical_intuition}
    In this paper, we restrict the selection of $\hjloss (\state)$ to be vector norms (e.g., p-norms, or weighted Q norms), and specifically, it is the same norm used in equation~\eqref{eqn:dst_to_set}. More specifically, given any POI $\poi \in \mathbb R^n$, we pick $\hjloss (\state ; \poi) = \|\state -\poi \|$ {(where $p$ is a hyperparameter)}. In other words, the loss function measures the distance of a state to the POI. With this restriction, the cost function \eqref{eqn:finite_time_costfunctional} captures the largest deviation from the POI of a given trajectory, initialized at $\state$ with $\csig$ and $\dsig$ applied, in time horizon $[\tinit,\thor]$. The infinite time value function \eqref{eqn:infinite_time_value_function} captures the largest deviation with optimal control and disturbance signals applied in an infinite time horizon. {Further, the limit in~\eqref{eqn:infinite_time_value_function} may not exist globally, but always exists on any compact subset of $\mathbb R^n$ under our assumptions.}
\end{remark}


{\begin{definition}(SRCIS.) Given a POI $\poi$, a norm, compute the value function $V^\infty(x)$ in~\eqref{eqn:infinite_time_value_function} and denote its minimal value as $\minval = \min_{x} V^\infty(x)$. The SRCIS, denoted as $\mrcis$, is the $\minval-$level set of $V^\infty$. 
\end{definition}}

{Note that any $\alpha-$sublevel set of $V^\infty$ is closed and bounded, therefore the minimal value $\minval$ is always attained. Further, since SRCIS is the level set corresponding to the minimal value, all states in it have the same value $\minval$. This also holds for the state $\bar x = \argmax_{x\in\mrcis} h(x;p)$. Because $\mrcis$ is robustly control invariant, the trajectory starting from $\bar x$ stays in $\mrcis$, and it's maximum deviation from $\poi$ is $\max_{x\in\mrcis} h(x;p)$. Combined, we have: }
\begin{align}\label{eqn:minval}
    \minval = \max_{\state \in  \mrcis} \hjloss (\state ; \poi).
\end{align}
In other words, any trajectory starting from $\mrcis$ has the same largest deviation \textbf{along time} (measured by $\hjloss$) to the POI, which is the value $\minval$. 

\begin{remark} \label{remark:levelsetCI}
    The SRCIS should be understood as `the RCIS, with the smallest largest deviation to the origin.'  Here the term \textit{smallest} means the `smallest level' of $V^\infty$, and  \textit{largest} means the largest deviation along time. Note that there may be other RCISs contained in the SRCIS, but they are not level sets of $V^\infty$. This is different from the `minimal RCIS' as defined in \cite{1406138,chen2018data}, where `minimal' is defined as `no subset is robust control invariant'.
    For illustration, consider the following example: 
\begin{align} \label{eqn:EX_IN2D}
    \dot x_1 = -x_1 +d, \quad \dot x_2 = x_2 + u
\end{align}
where $u\in[-1,1]$ and $d\in [-0.5,0.5]$. This system has an undisturbed, uncontrolled equilibrium point $[\state,\ctrl,\dstb] = [\textbf{0},0,0]$. It can be verified that $\rcis = \{ x_1\in [-0.5,0.5], x_2 = 0 \}$ is one `minimal RCIS' as all its subsets are not robustly control invariant. In fact, picking any $x_2 \in [-1,1]$ results in a `minimal RCIS.' On the other hand, picking $\hjloss (\state) = ||\state||_\infty$, the SRCIS is $\mrcis = \{ x_1,x_2 \in [-0.5,0.5] \}$. This is because though the control can stabilize any $|x_2| < 1$ to the origin, the disturbance is also strong enough to perturb any $|x_1|<0.5$ to leave the origin. Therefore, all states s.t. $x_1,x_2 \in [-0.5,0.5]$ have the same value, and the SRCIS measured by the $\infty$-norm is a square. Fig.~\ref{fig:IN2D_diff_cost} shows the SRCIS for three different choices of $\hjloss (\state)$ and the corresponding value function. 
\end{remark}

\begin{figure}[t]
\centering
\includegraphics[width=\columnwidth]{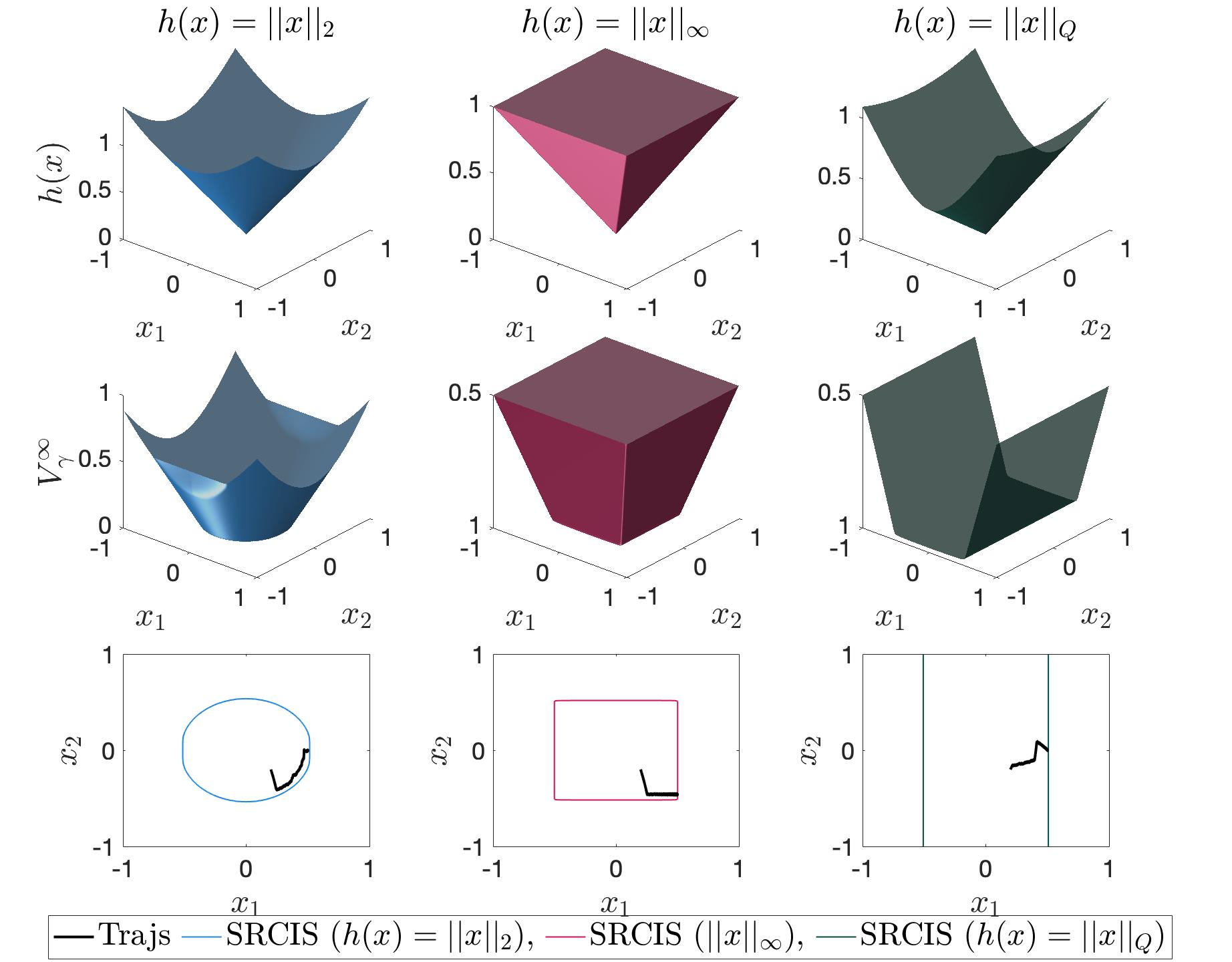}
    \caption{SRCIS corresponds to different loss functions for system \eqref{eqn:EX_IN2D}. Top left to right: different loss functions, including the 2-norm, infinity norm, and weighted Q-norm ($Q = diag [0.2, 1]$). Middle left to right: R-CLVF ($\gamma = 0$) when $h(\state) = ||\state||_2$, $||\state||_\infty$, $||\state||_Q$, and $||\state||_Q = \sqrt{\state ^T Q\state}$. Bottom left to right: the corresponding SRCIS and a trajectory starting inside the SRCIS. The robust control invariance is validated. 
    }
    \label{fig:IN2D_diff_cost}
\end{figure}


An interesting observation is that adding or subtracting a constant value to the loss function $h (\state)$, the corresponding SRCIS stays the same. 
{\begin{remark} \label{remark:moving_value_function}
Define $ \underline \hjloss (\state) = \hjloss (\state) - a$, and denote the corresponding value function as $\underline V(\state, \tinit)$, we have $\underline V(\state, \tinit) =  V(\state, \tinit) - a$. This is because 
\begin{align*} 
    \underline V(\state, \tinit) &= \sup _{\dmap \in \Dmap_\tinit} \inf _{u\in \cfset_t} \max _{s\in [\tinit,\thor]}  \underline \hjloss (\xi(\tvar;\tinit, x, \csig, \dmap[\ctrl]), \\ 
    &= \sup _{\dmap \in \Dmap_\tinit} \inf _{u\in \cfset_t} \max _{s\in [\tinit,\thor]} \big( h ( \xi(\tvar;\tinit, x, \csig, \dmap[\ctrl]) -a\big),   \\
 & = V(\state, \tinit) - a. 
\end{align*}
The last step is because $a$ is a constant. This holds for all $t\leq 0$, so we also have
\begin{align*}
    \underline V^\infty (x) = \lim_{t\rightarrow -\infty } \underline V(\state, \tinit)  = V^\infty(x) - a.
\end{align*}
\end{remark}}

Each level set of the HJ value function \eqref{eqn:infinite_time_value_function} is only robustly control invariant, there is no guarantee that the system can be stabilized to lower level sets or the origin. In our preliminary work \cite{ gong2022constructing}, we define the CLVF for undisturbed systems. In this article, we further develop the theory of CLVFs for disturbed systems and try to stabilize the system to any arbitrary POI. We also provide the necessary theorems for numerical implementation in high-dimensional nonlinear systems.   

%% file: assumptions.tex
We make the following assumptions about the system:
\begin{enumerate}[label=(\subscript{A}{{\arabic*}})]
    \item The dynamic model  $\dyn: \mathbb R^n \times \cset \times \dset \mapsto \mathbb R^n$ is uniformly continuous in $(\state,\ctrl,\dstb)$, Lipschitz continuous in $\state$ for fixed $\csig$ and $\dsig$, with Lipschitz constant $L_f$. \label{assumption 1}
    \item The dynamic model  $\dyn: \mathbb R^n \times \cset \times \dset \mapsto \mathbb R^n$ is bounded $\forall \state \in \mathbb R^n, \ctrl \in \cset, \dstb \in \dset$.\label{assumption 2}
\end{enumerate}

%% file: assumptions2.tex
\begin{enumerate}[label=(\subscript{A}{{\arabic*}})]
    \setcounter{enumi}{2}
    \item When the system has equilibrium points, the origin \textbf{0} is one, i.e. $\dyn( $\textbf{0,0,0}$) = $ \textbf{0}. 
    \item Given any point of interest (including the origin), there exists a robustly control invariant set, whose convex hull contains the point of interest.  
\end{enumerate}

%% file: proofs/proof_thrm_TV-DPP.tex
\begin{proof}
{Let $W(x, t)$ denotes the RHS of \eqref{eqn:TV-R-CLVF_DPP}. By definition of $W$, $\forall \dmap \in \Dmap_\tinit$, $\exists \bar u(\cdot) \in \cfset_\tinit $, s.t. for all $ \varepsilon_1 > 0$, 
\begin{align} \label{eqn:CLVF_DPP_ineq1}
    W(x,t) \geq \max \bigr \{ e^{\gamma \delta} \tclvf (\xi (t+\delta; t,x,\bar u, \lambda [\bar u]), t+\delta) ,\notag \\ 
    \max_{s\in[t,t+\delta]} e^{\gamma (s-t)} \ell (\xi(s; t,x,\bar u, \lambda [\bar u]))    \bigl \}  -\varepsilon_1.
\end{align}
Further, by equation~\eqref{eqn:finite_time_CLVF}, for all $\varepsilon>0$, $\exists \bar \lambda \in \Dmap_\tinit$ s.t. for all $\csig \in \cfset_\tinit$, 
\begin{align}\label{eqn:tvdpp_ineq2}
    \tclvf(x,t) \leq \max_{s\in[t,0]} e^{\gamma (s-t)} \ell (\xi(s; t,x, u, \bar \lambda [ u])) +\varepsilon.
\end{align}
Denote $y =\xi (t+\delta; t,x,\bar u, \lambda [\bar u]), t+\delta) $. Again, by equation~\eqref{eqn:finite_time_CLVF}, $\forall \lambda \in \Dmap_{t+\delta}$, $\exists \tilde  u \in \cfset_{t+\delta}$, s.t. $\forall \varepsilon_2>0$,
\begin{align*}
    \tclvf(y,t+\delta) \geq \max_{s\in[t+\delta,0]} e^{\gamma (s-t)} \ell (\xi(s; t,x, \tilde u,  \lambda [\tilde u])) - \varepsilon_2.
\end{align*}
Multiplying $e^{\gamma \delta}$ on both side, we get
\begin{align}\label{eqn:tvdpp_ineq3}
    e^{\gamma \delta}\tclvf(y,t+\delta) \geq \max_{s\in[t+\delta,0]} e^{\gamma (s-t)} \ell (\xi(s; t,x,\tilde u,  \lambda [\tilde u])) \notag \\ - e^{\gamma \delta}\varepsilon_2.
\end{align}
Define the control:
\begin{equation*}
    \hat{u}(s) := 
    \begin{cases}
        \bar u(s)\ &\text{if}\quad t \leq s < t+\delta,\\
        \tilde u(s)\ &\text{if}\quad t+\delta \leq s \leq 0,
    \end{cases}
\end{equation*}
and disturbance $\hat \lambda$ s.t. $\hat \lambda [\tilde u](x) = \bar \lambda [\tilde u](s)$ for all $s \in [t, t+\delta]$. Plugging~\eqref{eqn:tvdpp_ineq3} into equation~\eqref{eqn:CLVF_DPP_ineq1}: 
\begin{align*}
    W(x,t) \geq & \max \bigr \{ \max_{s\in[t+\delta,0]} e^{\gamma (s-t)} \ell (\xi(s; t,x,\hat u,  \lambda [\hat u])) - e^{\gamma \delta}\varepsilon_2,\notag \\ 
    &\max_{s\in[t,t+\delta]} e^{\gamma (s-t)} \ell (\xi(s; t,x,\hat u, \hat \lambda [\hat u])) \bigl \}  -\varepsilon_1 ,\notag \\
    \geq & \max \bigr \{ \max_{s\in[t+\delta,0]} e^{\gamma (s-t)} \ell (\xi(s; t,x,\hat u, \hat \lambda [\hat u])) ,\notag \\ 
    &\max_{s\in[t,t+\delta]} e^{\gamma (s-t)} \ell (\xi(s; t,x,\hat u, \hat \lambda [\hat u]))    \bigl \}  -\varepsilon_1 - e^{\gamma \delta}\varepsilon_2, \\
    =& \max_{s\in[t,0]} e^{\gamma (s-t)} \ell (\xi(s; t,x,\hat u, \hat \lambda [\hat u]))    \bigl \}  -\varepsilon_1 - e^{\gamma \delta}\varepsilon_2.
\end{align*}
Notice we have replaced $\bar u, \tilde u$ with $\bar u$ and $\bar \lambda$ with $\hat \lambda$ in the corresponding time intervals. Notice also that~\eqref{eqn:tvdpp_ineq2} still holds with $\hat u$ and $\hat \lambda$, therefore, we have
\begin{align}\label{eqn:tvdpp_1}
     W(x,t) \geq \tclvf(x,t) -\varepsilon_1 - e^{\gamma \delta}\varepsilon_2 - \varepsilon.
\end{align}
}

{On the other hand, by definition of $W$, $\forall \varepsilon_4 > 0$, $\exists \bar \dmap \in \Dmap_\tinit$ s.t. $\forall \csig \in \cfset_\tinit $, 
\begin{align} \label{eqn:tvdpp_ineq4}
    W(x,t) \leq \max \bigr \{ e^{\gamma \delta} \tclvf (\xi (t+\delta; t,x, u, \bar \lambda [ u]), t+\delta) ,\notag \\ 
    \max_{s\in[t,t+\delta]} e^{\gamma (s-t)} \ell (\xi(s; t,x, u, \bar \lambda [u]))    \bigl \}  + \varepsilon_4.
\end{align}
Similarily, denote $y =\xi (t+\delta; t,x, u, \bar \lambda [ u]), t+\delta) $. By~\eqref{eqn:finite_time_CLVF}, $\forall \varepsilon_5>0$, $\exists \tilde \lambda \in \Dmap_{\tinit +\delta}$, s.t. $\forall \csig \in \cfset_{\tinit+\delta}$, 
\begin{align} \label{eqn:tvdpp_ineq5}
    \tclvf(y,t+\delta) \leq \max_{s\in[t+\delta,0]} e^{\gamma (s-t-\delta)} \ell (\xi(s; t+\delta,y, u, \tilde \lambda [ u])) +\varepsilon_5.
\end{align}
Define the disturbance:
\begin{equation*}
    \hat{\dmap}(s) := 
    \begin{cases}
        \bar \dmap(s)\ &\text{if}\quad t \leq s < t+\delta,\\
        \tilde \dmap(s)\ &\text{if}\quad t+\delta \leq s \leq 0.
    \end{cases}
\end{equation*}
Plugging~\eqref{eqn:tvdpp_ineq5} into~\eqref{eqn:tvdpp_ineq4}, we get
\begin{align*}
    W(x,t) \leq \max \bigr \{ \max_{s\in[t,t+\delta]} e^{\gamma (s-t)} \ell (\xi(s; t,x, u, \bar \lambda [u])),   \\ \max_{s\in[t+\delta,0]} e^{\gamma (s-t)} \ell (\xi(s; t+\delta,y, u, \tilde \lambda [ u])) +e^{\gamma \delta} \varepsilon_5
     \bigl \}  + \varepsilon_4,
\end{align*}
which implies
\begin{align} \label{eqn:tvdpp_ineq6}
    W(x,t) \leq  \max_{s\in[t,0]} e^{\gamma (s-t)} \ell (\xi(s; t,x, u, \hat \lambda [u])) +e^{\gamma \delta} \varepsilon_5 + \varepsilon_4,
\end{align}
Again, by~\eqref{eqn:finite_time_CLVF}, $\forall \dmap \in \Dmap_\tinit$, $\exists \hat u \in \cfset_\tinit$ s.t.
\begin{align} \label{eqn:tvdpp_ineq7}
    \tclvf(x,t) \geq \max_{s\in[t,0]} e^{\gamma (s-t)} \ell (\xi(s; t,x, \hat u, \lambda [\hat u])).
\end{align}
Combining~\eqref{eqn:tvdpp_ineq6} and~\eqref{eqn:tvdpp_ineq7}, we have
\begin{align} \label{eqn:tvdpp_2}
    W(x,t) \leq \tclvf(x,t) +e^{\gamma \delta} \varepsilon_5 + \varepsilon_4.
\end{align}
From~\eqref{eqn:tvdpp_1} and~\eqref{eqn:tvdpp_2}, the proof is done.
}
\end{proof}

%% file: proofs/proof_sameZeroSet.tex
We only prove the first statement, as the second statement can be proved with the same process, and is an easy extension of the first statement. 
Given any $\state$, assume $V_{\gamma_1}(\state,\tinit) < 0$ and $V_{\gamma_2}(\state,\tinit) > 0$. Since $e^{\gamma_1 (\tvar - \tinit)} >0$, $\forall \dmap$, there must exist $\ctrl _1(\cdot)$ s.t. 
\begin{align*}
    \ell(\xi(t_1;\tinit,\state,\ctrl_1(\cdot),\dmap [\ctrl_1]) < 0 
\end{align*}
for some $t_1$. On the other hand, since $V_{\gamma_2}(\state,\tinit) > 0$ and $e^{\gamma_2 (\tvar - \tinit)} >0$, there exists $\bar \dmap $ and for all $\csig$ s.t. 
\begin{align*}
    \ell(\xi(t_1;\tinit,\state,\ctrl(\cdot),\bar \dmap [\ctrl]) \geq 0 
\end{align*}
for all $\tvar \in [\tinit,\thor]$. However, applying $\ctrl _1(\cdot)$ gives 
\begin{align*}
    \ell(\xi(t_2;\tinit,\state,\ctrl_1(\cdot),\bar \dmap [\ctrl_1]) < 0 
\end{align*}
for some $t_2$, which is a contradiction. Therefore, $V_{\gamma_1}(\state,t) < 0$ implies $V_{\gamma_2}(\state,t) \leq 0 $. Switching $\gamma_1$ and $\gamma_2$, we get the same result, and therefore we have $V_{\gamma_1}(\state,t) \leq 0 $ if and only if $V_{\gamma_1}(\state,t) \leq 0 $. With the same process, we could also show that $V^\infty_{\gamma_1}(\state,t) \leq 0 $ if and only if $V^\infty_{\gamma_1}(\state,t) \leq 0 $. Further, the above inequalities also hold when either $\gamma_1 $ or $\gamma_2$ is 0, and when $\gamma = 0$, in the infinite-time horizon, the zero sub-level set is the SRCIS. Therefore, we have proved $\mathcal Z_{\gamma}^\infty = \mrcis$.

%% file: proofs/proof_ExpimplyExists.tex
\begin{proof}
    Assume the system is exponentially stabilizable to the SRCIS. Using Definition \ref{def:ROES}, we have $\forall \dmap \in \Dmap $, $\exists u^*(\cdot) \in \cfset$, $\exists c > 0$ s.t. 
    \begin{align*}
        &dst(\opttraj;\mrcis) \leq c e^{-\gamma (\tvar-\tinit)} dst(\state;\mrcis). 
    \end{align*}
    Plugging in equation~\eqref{eqn:dst_to_set}, 
    \begin{align} \label{eqn:dst1}
        & \min _ {a \in \partial \mrcis}\|\opttraj-a\| \notag \\
        & \hspace{6em} \leq c e^{-\gamma (\tvar-\tinit)} \min _ {a \in \partial \mrcis}\|\state-a\|. 
    \end{align}
    Plugging in $\loss(\state) = \hjloss (\state) - \minval = \|\state\| -\minval$, we have
        \begin{align*}
         &\loss(\opttraj)  \\
         = & \hjloss(\opttraj) - \minval , \\
         = & \|\opttraj\| - \max _{a \in \partial \mrcis} \|a\|,  \\
         \leq & \|\opttraj\| - \min _{a \in \partial \mrcis} \|a\|,
         \end{align*}
         and 
         \begin{align}  \label{eqn:dst2}
         & \|\opttraj\| - \min _{a \in \partial \mrcis} \|a\|, \notag \\
         = &  \min _{a \in \partial \mrcis} \bigl( \|\opttraj\| -  \|a\|  \bigl) ,\notag \\
         \leq &  \min _{a \in \partial \mrcis} \bigl( \|\opttraj - a\|  \bigl), \notag \\
         \leq &  c e^{-\gamma (\tvar-\tinit)} \min _ {a \in \partial \mrcis}\|\state-a\|,
    \end{align}
    where we used equation~\eqref{eqn:dst1} for the last inequality. 
    Multiplying $e^{\gamma (s-t)}$ on both side
    \begin{align*}
        &e^{\gamma (s-t)} \bigl( \loss(\opttraj)  \bigl) \\ \leq & e^{\gamma (s-t)}  c e^{-\gamma (\tvar-\tinit)} \min _ {a \in \partial \mrcis}\|\state-a\|= c\min _ {a \in \partial \mrcis}\|\state-a\|, 
    \end{align*}
    which holds for all $\tvar \in [\tinit, \thor]$. Therefore
    \begin{align*}
        V_\gamma(\state,\tinit) = &\max_{\tvar \in [\tinit, \thor]}e^{\gamma (s-t)} \left( \loss(\opttraj) \right), \\
        \leq & c \min _ {a \in \partial \mrcis}\|\state-a\|.
    \end{align*}
    This upper bound $c \min _ {a \in \partial \mrcis}\|\state-a\|$ is independent of $\tinit$, therefore as $\tinit \rightarrow -\infty$, we have $\clvf (\state) \leq c\min _ {a \in \partial \mrcis}\|\state-a\|$. Since the R-CLVF monotonically increases, we conclude that the limit in \eqref{eqn:infinite_time_CLVF} exists $\forall \state \in \roes$. {This also implies $\roes \subset \dom$.}
    
\end{proof}

%% file: proofs/proof_thrm_DPP.tex
\begin{proof}
Denote the right-hand side of equation~\eqref{eqn:CLVF_DPP} as $W(\state)$. From the definition of R-CLVF, $\forall \varepsilon > 0$, $\state \in \dom$, $\exists \tinit \leq 0$ s.t.
\begin{align} \label{eqn:DPP_init}
    V_\gamma(\state,\tinit) \leq \clvf (\state) \leq V_\gamma (\state,\tinit) + \varepsilon.
\end{align}
From Theorem \ref{thrm:TV-R-CLVF_DPP}, $\forall \tinit < \tinit + \delta \leq 0$ and $\forall \state \in \dom$, define $z = \xi(\tinit+\delta;\tinit,\state,\csig,\dmap[\ctrl](\cdot))$, we have: 
\begin{align}\label{eqn:DPP_left}
     &\clvf(\state) \leq V_\gamma (\state , \tinit)  + \varepsilon \notag 
    = \sup _{\dmap \in \Dmap} \inf _{u \in \mathbb{U}} \max \biggl\{ \\ &e^{\gamma \delta} V_\gamma(\xi(t+\delta),t+\delta), \notag 
    \max_{s\in[t, t+\delta]} e^{\gamma (s-t)} \loss(\xi(s)) \biggl\}  + \varepsilon, \notag \\
    \leq &   \sup _{\dmap \in \Dmap} \inf _{u \in \mathbb{U}} \max \biggl\{e^{\gamma \delta} \clvf (z) , \notag \\  &\hspace{15mm}
    \max_{s\in[t, t+\delta]} e^{\gamma (s-t)} \loss(\xi(s)) \biggl\}  + \varepsilon 
    = W(\state) +\varepsilon.
\end{align}

On the other hand, using inequality~\eqref{eqn:DPP_init} and Theorem \ref{thrm:TV-R-CLVF_DPP}, $\forall \varepsilon >0$, $\forall \tinit < \tinit + \delta \leq 0$ and $\forall \state \in \dom$ we have: 
\begin{align}\label{eqn:DPP_right}
     &\clvf(\state) \geq V_\gamma (\state , \tinit)   \notag 
    = \sup _{\dmap \in \Dmap} \inf _{u \in \mathbb{U}} \max \biggl\{\\ &e^{\gamma \delta} V_\gamma(\xi(t+\delta),t+\delta), \notag 
    \max_{s\in[t, t+\delta]} e^{\gamma (s-t)} \loss(\xi(s)) \biggl\} , \notag \\
    \geq &   \sup _{\dmap \in \Dmap} \inf _{u \in \mathbb{U}} \max \biggl\{e^{\gamma \delta} (\clvf (z) - \varepsilon), \notag \\  &\hspace{25mm}
    \max_{s\in[t, t+\delta]} e^{\gamma (s-t)} \loss(\xi(s)) \biggl\}  , \notag \\
    \geq &  \sup _{\dmap \in \Dmap} \inf _{u \in \mathbb{U}} \max \biggl\{e^{\gamma \delta} \clvf (z) , \notag \\  &\hspace{5mm}
    \max_{s\in[t, t+\delta]} e^{\gamma (s-t)} \loss(\xi(s)) \biggl\}  - e^{\gamma \delta} \varepsilon
     = W(\state) - e^{\gamma \delta} \varepsilon.
\end{align}

Combining equation~\eqref{eqn:DPP_left} and~\eqref{eqn:DPP_right}, we show $\forall \varepsilon >0$, $\forall \tinit < \tinit + \delta \leq 0$ and $\forall \state \in \dom$:
\begin{align*}
      W(\state) -e^{\gamma \delta}\varepsilon \leq \clvf(\state) \leq W(\state) +\varepsilon,
\end{align*}
which completes the proof.

\end{proof}

%% file: proofs/proof_infiniteVS.tex
\begin{proof}
First, define $\mathcal F(\state,v,p): \dom \times \mathbb R \times \mathbb R^n \mapsto \mathbb R$
\begin{align*}
    \mathcal F(\state,v,p) = \max \{ \loss(\state) - v, \hspace{.5em} H(\state,v,p) \},
\end{align*}
then the R-CLVF-VI can be written as 
\begin{align*}
    \mathcal F(\state,\clvf(\state), D_x \clvf(\state)) = 0.
\end{align*}
Following~\cite{barron1989bellman}, 
\begin{enumerate}
    \item $\clvf $ is a viscosity \textbf{sub-}solution of~\eqref{eqn:CLVF-VI}, if for any $\phi \in C^1(\dom)$, and $x$ is a local maxima for $V_{\gamma}^\infty - \phi$, 
    \begin{align}\label{eqn:VS_sub}
        \mathcal F(\state,\clvf(\state), D_x \phi (\state)) \geq 0.
    \end{align}
     \item $\clvf$ is a viscosity \textbf{super-}solution of~\eqref{eqn:CLVF-VI}, if for any $\psi \in C^1(\dom)$, and $x$ is a local minima for $V_{\gamma}^\infty - \psi$, 
    \begin{align}\label{eqn:VS_super}
        \mathcal F(\state,\clvf(\state), D_x \psi (\state)) \leq 0.
    \end{align}
\end{enumerate}

{Without loss of generality,} we could always assume $\state$ is a strict local maxima (minima), and the maximum (minimum) value is $0$, i.e., $\clvf(\state) - \phi(\state) = 0$ ($\clvf(\state) - \psi(\state) = 0$). Here, since the R-CLVF is the limit function of TV-R-CLVF, we use the definition for the TVP. 

We start with the sub-solution. Assume~\eqref{eqn:VS_sub} is wrong, i.e., 
\begin{align*}
        \mathcal F(\state,\clvf(\state), D_x \phi (\state)) < 0.
\end{align*}
Then there exists $\theta_1, \theta_2 > 0$ s.t. both the followings hold 
\begin{align}
    &\loss(\state) - \clvf(\state) \leq -\theta_1 < 0, \label{eqn:VS_sub1} \\
    & \max _{\dstb \in \dset} \min _{ \ctrl \in \cset} D_x \phi(x) \cdot \dyn(\state, \ctrl, \dstb) + \gamma V_{\gamma}^\infty(x) \leq -\theta_2 < 0. \label{eqn:VS_sub2}
\end{align}
By continuity of $\loss$ and $\xi$, there exists $\delta_1 > 0$ s.t. for any $ \theta_1 > 0$, $\csig, \dmap$, $\tvar \in [\tinit, \tinit + \delta_1]$ 
\begin{align} \label{eqn:VS_l-V}
    | e^{\gamma (\tvar - \tinit)} \loss (\trajmap) - \loss (\state) | \leq \frac{\theta_1}{2}.
\end{align}
Combined with~\eqref{eqn:VS_sub1}, we have
\begin{align} \label{eqn:VS_sub3}
    e^{\gamma (\tvar - \tinit)} \loss (\trajmap) \leq  \loss (\state) + \frac{\theta_1}{2} 
    \leq  \clvf(\state) -\frac{\theta_1}{2}.
\end{align}
Further, since $\clvf(\state) = \phi (\state)$, ~\eqref{eqn:VS_sub2} can be written as
\begin{align*}
     \max _{\dstb \in \dset} \min _{ \ctrl \in \cset} D_x \phi(x) \cdot \dyn(\state, \ctrl, \dstb) + \gamma \phi(\state)  \leq -\theta_2.
\end{align*}
and for any $ \dstb \in \dset$,
\begin{align*}
     \min _{ \ctrl \in \cset} D_x \phi(x) \cdot \dyn(\state, \ctrl, \dstb) + \gamma \phi(\state)  \leq -\theta_2.
\end{align*}
Since $\phi \in C^1$, there exists $\delta > 0, \bar \ctrl \in \cset$, $\tvar \in [\tinit, \tinit + \delta]$ s.t. 
\begin{align*}
      D_x \phi(\xi(\tvar)) \cdot \dyn(\xi(\tvar), \bar \ctrl(\tvar), \dmap[ \bar \ctrl](\tvar)) + \gamma \phi(\state(\tvar))  \leq -\frac{\theta_2}{2}.
\end{align*}
Multiplying $e^{\gamma(\tvar-\tinit)}$ on both side and integrating on both side for $\tvar \in [\tinit, \tinit+\delta]$, we get:
\begin{align*}
    e^{\gamma \delta} \phi (\xi(\tinit+\delta;x,t,\bar \ctrl (\cdot), \dmap[\bar \ctrl])) -\phi(\state) \leq -\frac{\theta_2 (e^{\gamma \delta} - 1)}{2\gamma}.
\end{align*}
Since $\clvf- \phi $ has local maxima at $\state$ and the maximum value is $0$, we have
\begin{align}
    &e^{\gamma \delta} \clvf (\xi(t+\delta;x,t,\bar \ctrl (\cdot), \dmap[\bar \ctrl]))  \notag \\ 
    \leq  &\phi(\state)  -\frac{\theta_2 (e^{\gamma \delta} - 1)}{2\gamma}
    =  \clvf (\state)  -\frac{\theta_2 (e^{\gamma \delta} - 1)}{2\gamma}.  \label{eqn:VS_sub4}
\end{align}
Combining~\eqref{eqn:VS_sub3} and~\eqref{eqn:VS_sub4}, and because $d$ is arbitrary, we have
\begin{align*}
     &\clvf (\state)  - \min \bigl\{ \frac{\theta_1}{2}, \frac{\theta_2 (e^{\gamma \delta} - 1)}{2\gamma} \bigr\} \\
      \geq & \sup_{\dmap} \max \bigl\{ e^{\gamma (\tvar - \tinit)} \loss (\trajmap), \\
      & \hspace{6em} e^{\gamma \delta} \clvf (\xi(t+\delta;x,t,\bar \ctrl (\cdot), \dmap[\bar \ctrl])) \bigr\} \\
      \geq & \clvf(\state),
\end{align*}
where the last inequality is from the DPP~\eqref{eqn:CLVF_DPP}. This is clearly a contradiction, so~\eqref{eqn:VS_sub} must hold, i.e., $\clvf$ is a viscosity sub-solution.

Now, we examine the super-solution. Assume~\eqref{eqn:VS_super} is wrong, i.e., 
\begin{align*}
        \mathcal F(\state,\clvf(\state), D_x \psi (\state)) > 0.
\end{align*}
Then $\exists \theta_1, \theta_2 > 0$ s.t. one of the followings hold 
\begin{align}
    &\loss(\state) - \clvf(\state) \geq \theta_1 > 0, \label{eqn:VS_super1} \\
    & \max _{\dstb \in \dset} \min _{ \ctrl \in \cset} D_x \psi(x) \cdot \dyn(\state, \ctrl, \dstb) + \gamma V_{\gamma}^\infty(x) \geq \theta_2 > 0. \label{eqn:VS_super2}
\end{align}
If~\eqref{eqn:VS_super1} holds, then from~\eqref{eqn:VS_l-V}, there exists $\delta _1 >0$ s.t. for any $\theta_1>0, \csig, \dmap$ and $\tvar \in [\tinit, \tinit+\delta _1]$
\begin{align*} 
    e^{\gamma (\tvar - \tinit)} \loss (\trajmap) \geq  \loss (\state) - \frac{\theta_1}{2} 
    \geq  \clvf(\state) +\frac{\theta_1}{2}.
\end{align*}
Plugging in the DPP~\eqref{eqn:CLVF_DPP}, we have 
\begin{align*}
    \clvf(\state) \geq &\sup_{\dmap} \inf_{\csig} \max_{s\in[\tinit, \tinit+ \delta]}e^{\gamma (\tvar - \tinit)} \loss (\trajmap) \\
    \geq& \clvf(\state) + \frac{\theta_1}{2},
\end{align*}
which is a contradiction. Therefore~\eqref{eqn:VS_super1} cannot be true. 

If~\eqref{eqn:VS_super2} holds, then from the same derivation of~\eqref{eqn:VS_sub4}, there exists $\delta_2 > 0$, $\exists \dmap, \forall \csig$, s.t. $\forall \tvar \in [\tinit, \tinit+\delta_2]$
\begin{align*}
    &e^{\gamma \delta} \clvf (\xi(t+\delta;x,t,\bar \ctrl (\cdot), \dmap[\bar \ctrl]))   \\ 
    \geq  &\psi(\state)  +\frac{\theta_2 (e^{\gamma \delta} - 1)}{2\gamma}
    =  \clvf (\state)  +\frac{\theta_2 (e^{\gamma \delta} - 1)}{2\gamma}.  
\end{align*}
Again plugging in DPP~\eqref{eqn:CLVF_DPP}, we have
\begin{align*}
    \clvf(\state) \geq &\sup_{\dmap} \inf_{\csig} e^{\gamma \delta} \clvf (\xi(t+\delta;x,t,\bar \ctrl (\cdot), \dmap[\bar \ctrl])),  \\
    \geq& \clvf(\state) + \frac{\theta_2 (e^{\gamma \delta} - 1)}{2\gamma},
\end{align*}
which is a contradiction. Therefore~\eqref{eqn:VS_super2} cannot be true. Combined,~\eqref{eqn:VS_super} must hold, so $\clvf$ is a viscosity super-solution. 

{We do not have the uniqueness guarantee for two reasons: 1) we do not specify any boundary conditions, and 2) the spatial growth of the R-CLVF is not bounded, as shown in Proposition~\ref{prop: Radially_unbounded}. Further, as will be seen in the next paragraph, it is not trivial to specify the boundary condition. }

\end{proof}

%% file: proofs/proof_ExiimplyExp.tex
\begin{proof}
Assume the limit in (\ref{eqn:infinite_time_CLVF}) exists in $\dom$. For any initial state $\state \in \dom \setminus \mrcis$, 
from Proposition \ref{prop:Vdot<=-gammaV}, {because $\cset$ is compact, there exists $u^*$ s.t. for all $d$,}
 \begin{align*}
     D_x \clvf(\state) \cdot f(\state, \ctrl^* ,\dstb) = \dot V_{\gamma}^\infty  \leq - \gamma V^\infty_{\gamma}.
\end{align*} 
Using the comparison principle, we have $\forall s \in [t,0]$,
 \begin{align} \label{eqn:EX1}
     \clvf \big( \opttraj \big)\leq e^{-\gamma (\tvar-\tinit) } \clvf (\state),
\end{align} 
{where $\csigopt$ is constructed by $u^*$ and $\dmap$ is arbitrary. Consider an arbitrary compact subset of $\dom $, denoted as $\mathcal S_\gamma$. Define
\begin{align*}
    \phi(x) : = \frac{\clvf(x)}{\text{dst}(x;\mrcis)}.
\end{align*}
Since $\mathcal S_\gamma$ is compact and both $\clvf$ and $dst(x;\mrcis)$ are continuous in $x$, $\phi$ attains its maximum and minimum in $\mathcal S_\gamma$. This means there exists $c_1>0$ and $c_2>0$ s.t.
\begin{align*}
    c_1\text{dst}(x;\mrcis) \leq \clvf(x) \leq c_2 \text{dst}(x;\mrcis).
\end{align*}
Combined with~\eqref{eqn:EX1}, we have
\begin{align*}
    &\text{dst}(\opttraj; \mrcis), \\
    \leq &\frac{1}{c_1} \clvf(\opttraj), \\
    \leq &\frac{1}{c_1}  e^{-\gamma (\tvar-\tinit) } \clvf (\state), \\
    \leq & \frac{c_2}{c_1}  e^{-\gamma (\tvar-\tinit) } {\text{dst}(x;\mrcis)}.
\end{align*}
Therefore, $\mathcal S_\gamma$ is a subset of $\roes$. Further, since $\mathcal S_\gamma$ is an arbitrary compact subset of $\dom$, we conclude that $\dom \subset \roes$.}

\end{proof}

%% file: proofs/proof_thrm_IEwarmstart.tex
\begin{proof}
We show this results for three cases: $k(\state) =\loss(\state)$, $k(\state) > \loss(\state)$, and $k(\state) < \loss(\state)$.

(1) Assume $k(\state) = \loss(\state)$. Then $\bar {V}_\gamma(\state,\tinit) \geq  {V_\gamma}(\state,\tinit)$  holds $\forall \state$ and $\tinit$.
        
(2) Assume $k(\state) > \loss(\state)$. From \eqref{eqn:TV-R-CLVF_DPP}, $\forall \tinit < \tinit+\delta = \thor$: 
        \begin{align*}
            \bar V_\gamma(\state,\tinit) 
            =& \sup _{\dmap \in \Dmap} \inf _{\ctrl \in \cfset} \max \biggl\{e^{ \gamma \delta} \bar {V}_\gamma(z,\thor), 
            \max _{\tvar \in [\tinit, \thor]} e^{\gamma (\tvar-\tinit)} \loss(\xi(s)) \biggl\} \\
            =& \sup _{\dmap \in \Dmap} \inf _{\ctrl \in \cfset} \max \biggl\{e^{\gamma \delta} k(\xi(0)), 
             \max _{\tvar \in [\tinit, \thor]} e^{\gamma (\tvar-\tinit)} \loss(\xi(s)) \biggl\} \\
            \geq & \sup _{\dmap \in \Dmap} \inf _{\ctrl \in \cfset} \max \biggl\{e^{\gamma \delta} \loss(\xi(0)), 
            \max _{\tvar \in [\tinit,\thor]} e^{\gamma (\tvar-\tinit)} \loss(\xi(s)) \biggl\} \\
            = & V_\gamma(x,t)
        \end{align*}

(3) Assume $k(\state) < \loss(\state)$. Then, at time $t = \thor$, we have $\bar {V}_\gamma(\state,\thor) < V_\gamma(\state,\thor)$. Consider an infinitesimal time step $\delta t<0$ and $\delta t \rightarrow \thor^-$, using \eqref{eqn:CLVF_DPP}, we have: 
         \begin{align*}
            \bar V_\gamma(\state,\thor^-) 
            =&  \sup _{\dmap \in \Dmap} \inf _{\ctrl \in \cfset} \max \biggl\{e^{\gamma 0^-} k(\xi(\thor^-)), \\
            & \hspace{5em}
            \max _{s\in[\thor^-, \thor]} e^{\gamma (s-\thor^-)} \loss(\xi(s) \biggl\} \\
            =&  \max \biggl\{e^{\gamma t_1} k(\xi(\thor)), 
             e^{-\gamma \thor^-} \loss(\xi(\thor^-) \biggl\} \\
            =& e^{-\gamma \thor^-} \loss(\xi(\thor^-)\\
            \geq & \loss(\xi(\thor^-) = V_\gamma(x,\thor^-),
        \end{align*}
        in other words, after one infinitesimal small step, we get $\bar V_\gamma(\state,\tinit ^-) >  V_\gamma((\state,\tinit ^-)$. Now, replace $k(\state) = \bar V_\gamma(\state,\tinit ^-) $, we return to the second case, and the remaining proof follows.

    \end{proof}

%% file: proofs/proof_thrm_Ewarmstart.tex
\begin{proof}
        Denote $\tilde k(\state) =\clvf(x)$, and the value function initialized with $\tilde k(\state)$ as $\tilde {V}_\gamma (\state,\tinit)$. we have $\forall \state, \tinit < \tinit+\delta = 0$: 
        \begin{align*}
            \tilde V_\gamma(\state,\tinit) 
            =& \sup _{\dmap \in \Dmap} \inf _{\ctrl \in \cfset} \max \biggl\{e^{-\gamma \tinit} \tilde {V}_\gamma(z,\thor), 
            \max _{\tvar\in[\tinit, \thor]} e^{\gamma (\tvar - \tinit)} \loss(\xi(\tvar)) \biggl\} \\
            =&  \sup _{\dmap \in \Dmap} \inf _{\ctrl \in \cfset} \max \biggl\{e^{-\gamma \tinit} \tilde k(\xi(\thor)), 
           \max _{\tvar\in[\tinit, \thor]} e^{\gamma (\tvar - \tinit)} \loss(\xi(\tvar)) \biggl\} \\
            \geq & \sup _{\dmap \in \Dmap} \inf _{\ctrl \in \cfset} \max \biggl\{e^{-\gamma \tinit} k(\xi(\thor)), 
           \max _{\tvar\in[\tinit, \thor]} e^{\gamma (\tvar - \tinit)} \loss(\xi(\tvar)) \biggl\} \\
            = & \bar {V}_\gamma(\state,\tinit).
        \end{align*}
        Note that $\clvf(\state)$ is the already the converged value function, we have $\clvf(\state)  = \tilde {V}_\gamma ^\infty (\state, \tinit) \geq  \bar {V}_\gamma(\state,\tinit)$. 
        
        Similar to Propsition \ref{prop: Vk>V}, If $ \clvf(\state)$ exists on $ \dom$, then $\bar {V}_\gamma^\infty(x) \leq  {V_\gamma^\infty}(x)$, and $\dom \subseteq \bar {\mathcal D}_\gamma$. Combined, we get $\dom = \bar {\mathcal D}_\gamma$, and $\forall \state \in \dom$, $\bar {V}_\gamma^\infty(\state) =\clvf(\state)$.
       
    \end{proof}

%% file: algo_compute.tex
\begin{algorithm}[t]
\caption{Obtaining the R-CLVF with warmstarting}
\label{Algorithm:CLVF}
\begin{algorithmic}[1]
 \Require: $\dyn(\state,\ctrl,\dstb)$, $\cset$, $\dset$, $\gamma >0$, convergence threshold $\Delta$, $\ell(x)$,  $\delta t$. \\
\textbf{Output}: $\clvf(\state)$, $\mrcis$\\
\textbf{Initialization:} 
 \State $ V(x,t_0) \gets \ell(x)$\\
 \textbf{Find $\mrcis$}
 \State $ V^\infty (x) \gets$ update\_value($\dyn$, $\cset$, $\dset$,  $\Delta$, $\delta t$, $V(x,0)$, $\ell(x)$)
 \State $\minval \gets \min_{\state} V^\infty(\state)$, $\mrcis \gets \{ V^\infty(x) = \minval\}$\\
 \textbf{Find R-CLVF}
 \State $\ell(x) \gets \ell(x) - \minval$, $V(x,t_0) \gets V^\infty(\state)  - \minval $
 \State $ \clvf (x) \gets$ update\_value($\dyn$, $\cset$, $\dset$,  $\Delta$, $\delta t$, $V(x,0)$, $\ell(x)$) \\
 \textbf{update\_value($\dyn$, $\cset$, $\dset$,  $\Delta$, $\delta t$, $V(x,0)$, $\ell(x)$)}
 \State $t\gets 0$
 \While {$dV \geq \Delta$}
 \State $V(x,t+\delta t) \gets V(x,t)$
 \State update $V(x,t+\delta t)$ using equations~\eqref{eqn:TV-R-CLVF_DPP}~\eqref{eqn:TV-R-CLVF-VI}
 \State $dV = \min _{\state} (V(x,t+\delta t) - V(x,t) )$
 \State $t\gets t+\delta t$
 \EndWhile

\end{algorithmic}
\end{algorithm}

%% file: proofs/proof_prop_TV_lipschitz.tex
{Given any open set $\mathcal C$ and arbitrary control and disturbance signals. Since the solution of~\eqref{eqn:dynamic_system} is guaranteed to exist given assumptions~\ref{assumption 1} and~\ref{assumption 2}, 
the cost function 
\begin{align}\label{eqn:proof_TV_CLVF_lip1}
    J_\gamma(\tinit,\state,\csig,\dmap) & = \max_{\tvar \in [\tinit, \thor]} e^{\gamma (\tvar - \tinit)} \loss (\trajmap), \notag \\
    & \leq  e^{ -\gamma  \tinit } \max_{\tvar \in [\tinit, \thor]} \loss (\trajmap)
\end{align}
is bounded. Since this holds for arbitrary control and disturbance signals, the TV-R-CLVF is also bounded. }

{For the local Lipschitzness in $\state$, given any $x, y \in \mathcal C$, $\csig \in \cfset, \dmap \in \Dmap$, we have: 
\begin{align*}
    & \| J_\gamma (\tinit,x,u(\cdot),\dmap) - J_\gamma (\tinit,y,u(\cdot),\dmap) \|\\
    =&\| \max _{\tvar \in [\tinit, \thor]} e^{\gamma (\tvar-\tinit)} \ell\big(\xi(s;t,x,u(\cdot),\dmap)\big) - \\
    & \hspace{3em} \max _{\tvar \in [\tinit, \thor]} e^{\gamma (\tvar-\tinit)}\ell \big(\xi(s;t,y,u(\cdot),\dmap) \big) \|\\ 
    \leq & \max _{\tvar \in [\tinit, \thor]} \| e^{\gamma (\tvar-\tinit)} \ell\big(\xi(s;t,x,u(\cdot),\dmap)\big)  -\\
    & \hspace{3em} e^{\gamma (\tvar-\tinit)} \ell\big(\xi(s;t,y,u(\cdot),\dmap)\big) \| \\
    \leq & \max _{\tvar \in [\tinit, \thor]}  e^{\gamma (s-t)} \|  \ell\big(\xi(s;t,x,u(\cdot),\dmap)\big)  -\\
    & \hspace{3em} \ell\big(\xi(s;t,y,u(\cdot),\dmap)\big) \| \\
    =& e^{ - \gamma \tinit} \|  \ell\big(\xi(s;t,x,u(\cdot),\dmap)\big)  -\\
    & \hspace{3em} \ell\big(\xi(s;t,y,u(\cdot),\dmap)\big) \|.
\end{align*}
Plug in the definition of $\loss$, we have: 
\begin{align*}
     & \| J_\gamma (\tinit,x,u(\cdot),\dmap) - J_\gamma (\tinit,y,u(\cdot),\dmap) \|\\
    \leq & e^{ - \gamma \tinit} \bigl \| \|  \xi(s;t,x,u(\cdot),\dmap) \|  - \| \xi(s;t,y,u(\cdot),\dmap) \| \bigr \| \\
    \leq & e^{ - \gamma \tinit} \|  \xi(s;t,x,u(\cdot),\dmap) - \xi(s;t,y,u(\cdot),\dmap) \| .
\end{align*}
Since this holds for arbitrary $\csig$ and $\dmap$, we have
\begin{align*}
     & \| \tclvf (x,t) - \tclvf(y,t)\|\\
    \leq & e^{ - \gamma \tinit} \|  \xi(s;t,x,u(\cdot),\dmap) - \xi(s;t,y,u(\cdot),\dmap) \|.
\end{align*}
Because of the continuous dependence on the initial condition, $\forall x,y \in \mathcal C$, there exists a constant $c>0$ such that 
\begin{align*}
     \|\xi(s;t,x,u(\cdot),\dmap) - \xi(s;t,y,u(\cdot),\dmap)\| \leq  c\|x-y\|,
\end{align*}
Combined, we have
\begin{align}\label{eqn:proof_TV_CLVF_lip2}
    \|\tclvf(x,t) -  \tclvf(y,t) \| \leq e^{ - \gamma \tinit}  c\|x-y\| 
\end{align}}



{For the local Lipschitzness in $\tinit$, for any $\tinit < t_1 \leq \tvar \leq \thor$, and any $\dmap, \csig$, we have
\begin{align*}
    J_\gamma(t_1,\state,\csig,\dmap) \geq e^{\gamma (\tvar - t_1)} \loss (\xi(\tvar ; t_1,\state, \csig, \dmap)).
\end{align*}
Combined with~\eqref{eqn:proof_TV_CLVF_lip1}, we have
\begin{align*}
    & J_\gamma(t,\state,\csig,\dmap) - J_\gamma(t_1,\state,\csig,\dmap) \\
    \leq & e^{ -\gamma  \tinit } \max_{\tvar \in [\tinit, \thor]} \loss (\trajmap) - \\
    & \hspace{6em } e^{\gamma (\tvar - t_1)} \loss (\xi(\tvar ; t_1,\state, \csig, \dmap)).
\end{align*}
From~\eqref{eqn:TV-R-CLVF_DPP}, we have
\begin{align*}
    & \tclvf(\state, \tinit) = \sup _{\dmap \in \Dmap_\tinit} \inf _{\csig \in \cfset_\tinit} \max \biggl\{e^{\gamma (t_1-t)} V_\gamma(\xi(t_1),t_1),  \\
    & \hspace{6em }\max_{s\in[t, t_1]} e^{\gamma (s-t)} \loss(\xi(s)) \biggl\},
\end{align*}
which means both of the following hold: 
\begin{align*}
     &\tclvf(\state, \tinit) \geq  e^{\gamma (t_1-t)} V_\gamma(\xi(t_1),t_1), \\
     &\tclvf(\state, \tinit) \geq \sup _{\dmap \in \Dmap_\tinit} \inf _{\csig \in \cfset_\tinit} \max_{s\in[t, t_1]} e^{\gamma (s-t)} \loss(\xi(s)) .
\end{align*}
The first inequality and \eqref{eqn:proof_TV_CLVF_lip2} implies 
\begin{align*}
    \tclvf(\state, \tinit) &\geq e^{\gamma (t_1-t)} V_\gamma(\xi(t_1),t_1) \\
    &\geq  V_\gamma(\xi(t_1),t_1) \\
    &\geq V_\gamma(x,t_1) - e^{-\gamma t_1} c\| x-\xi(t_1) \|,
\end{align*}
therefore we have
\begin{align*}
    \| \tclvf(\state, \tinit) - \tclvf(\state, t_1) \|  & \leq e^{-\gamma t_1} c\| x-\xi(t_1) \| \\
    & \leq e^{-\gamma t_1} c L_f \| t_1 - t \|. 
\end{align*}}

%% file: proofs/proof_thrm_TV_VS_arxiv.tex
\begin{proof}
{Following~\cite{barron1989bellman}, 
\begin{enumerate}
    \item $\tclvf $ is a viscosity \textbf{sub-}solution of~\eqref{eqn:TV-R-CLVF-VI}, if for any $\phi \in C^1(\mathbb R^n \times (-\infty,0])$, $(x_0,t_0) \in \mathbb R^n \times (-\infty,0]$ is a local maxima for $\tclvf - \phi$, then
    \begin{align}\label{eqn:TV_VS_sub}
        &0 \leq \max\biggl\{\loss(x_0) - \tclvf (x_0,t_0), \hspace{1em} D_\tinit \phi(x_0,t_0)+ \notag \\
         &\max_{\dstb \in \dset } \min_{ \ctrl \in \cset} D_\state \phi(x_0,t_0) \cdot \dyn(x_0, \ctrl, \dstb) + \gamma \tclvf(x_0,t_0)  \biggl\},
    \end{align}
     \item $\tclvf$ is a viscosity \textbf{super-}solution of~\eqref{eqn:CLVF-VI}, if for any $\psi \in C^1(\mathbb R^n \times (-\infty,0])$, $(x,t) \in \mathbb R^n \times (-\infty,0]$ is a local minima for $\tclvf - \psi$, then
    \begin{align}\label{eqn:TV_VS_super}
        &0 \geq \max\biggl\{\loss(x_0) - \tclvf (x_0,t_0), \hspace{1em} D_\tinit \psi(x_0,t_0)+ \notag \\
         &\max_{\dstb \in \dset } \min_{ \ctrl \in \cset} D_\state \psi(x_0,t_0) \cdot \dyn(x_0, \ctrl, \dstb) + \gamma \tclvf(x_0,t_0)  \biggl\},
    \end{align}
\end{enumerate}
W.L.O.G, we could always assume $\state$ is a strict local maxima (minima), and the maximum (minimum) value is $0$, i.e., $\tclvf(\state_0,t_0) - \phi(\state,t_0) = 0$ ($\tclvf(\state_0,t_0) - \psi(\state_0,t_0) = 0$). Finally, $\tclvf$ is a viscosity solution of~\eqref{eqn:TV-R-CLVF-VI} if it is both a sub-solution and a supersolution. }

{We will prove by contradiction. We start with the sub-solution. Assume~\eqref{eqn:TV_VS_sub} is wrong, then there exists $\theta_1, \theta_2 > 0$ s.t. both the followings hold 
\begin{align}
    &\loss(\state_0) - \tclvf(\state_0,\tinit_0) \leq -\theta_1 < 0, \label{eqn:TV_VS_sub1} \\
    &  D_\tinit \phi(x_0,t_0) + \max _{\dstb \in \dset} \min _{ \ctrl \in \cset} D_x \phi(x_0,t_0) \cdot \dyn(\state_0, \ctrl, \dstb) \notag \\
    & + \gamma \tclvf(x_0,t_0) \leq -\theta_2 < 0. \label{eqn:TV_VS_sub2}
\end{align}
By continuity of $\loss$ and $\xi$, there exists $\delta_1 > 0$ s.t. for any $\csig, \dmap[\ctrl]$, $\tvar \in [\tinit_0, \tinit_0 + \delta_1]$ 
\begin{align} \label{eqn:TV_VS_l-V}
    | e^{\gamma (\tvar - \tinit_0)} \loss (\xi (s;t_0,x_0,u,\dmap [\ctrl])) - \loss (\state_0) | \leq \frac{\theta_1}{2}.
\end{align}
Combined with~\eqref{eqn:TV_VS_sub1}, we have
\begin{align} \label{eqn:TV_VS_sub3}
    e^{\gamma (\tvar - \tinit_0)} \loss (\xi(s;t_0,x_0,u(\cdot),\dmap)) &\leq  \loss (\state_0) + \frac{\theta_1}{2} , \notag \\
    &\leq  \tclvf(\state_0,\tinit_0) -\frac{\theta_1}{2}. 
\end{align}
Further, since $\tclvf(\state_0,\tinit_0) = \phi (\state_0,\tinit_0)$, ~\eqref{eqn:TV_VS_sub2} can be written as
\begin{align*}
      D_\tinit \phi(x_0,t_0) + \max _{\dstb \in \dset} \min _{ \ctrl \in \cset} D_x \phi(x_0,t_0) \cdot \dyn(\state_0, \ctrl, \dstb) + \notag \\
    \gamma \phi(x_0,t_0) \leq -\theta_2
\end{align*}
and for any $ \dstb \in \dset$,
\begin{align*}
     D_\tinit \phi(x_0,t_0) +\min _{ \ctrl \in \cset} D_x \phi(x_0,t_0) \cdot \dyn(\state_0, \ctrl, \dstb) + \gamma \phi(\state_0,t_0)  \\ \leq -\theta_2.
\end{align*}
Since $\phi \in C^1$, there exists $\delta > 0, \bar \ctrl \in \cset$, $\tvar \in [\tinit_0, \tinit_0 + \delta]$ s.t. 
\begin{align*}
       D_\tinit \phi(\xi(s),s) + D_x \phi(\xi(\tvar),s) \cdot \dyn(\xi(\tvar), \bar \ctrl(\tvar), \dmap[ \bar \ctrl](\tvar)) \\ + \gamma \phi(\xi(\tvar),s)
       \leq -\frac{\theta_2}{2}.
\end{align*}
Muliplly both side with $e^{\gamma(\tvar-\tinit_0)}$ and integrate on both side for $\tvar \in [\tinit_0, \tinit_0+\delta]$, we get:
\begin{align*}
    e^{\gamma \delta} \phi (\xi(t_0+\delta;x_0,t_0,\bar \ctrl (\cdot), \dmap[\bar \ctrl]), t_0+\delta) -\phi(\state_0,t_0) \\
    \leq -\frac{\theta_2 (e^{\gamma \delta} - 1)}{2\gamma}.
\end{align*}
Since $\tclvf- \phi $ has local maxima at $(\state_0,t_0)$ and the maximum value is $0$, we have
\begin{align}
    &e^{\gamma \delta} \tclvf (\xi(t_0+\delta;x,t,\bar \ctrl (\cdot), \dmap[\bar \ctrl]),t_0+\delta)  \notag \\ 
    \leq  &\phi(\state_0,t_0)  -\frac{\theta_2 (e^{\gamma \delta} - 1)}{2\gamma}
    =  \tclvf (\state_0,t_0)  -\frac{\theta_2 (e^{\gamma \delta} - 1)}{2\gamma}.  \label{eqn:TV_VS_sub4}
\end{align}
Combine~\eqref{eqn:TV_VS_sub3} and~\eqref{eqn:TV_VS_sub4}, and because $d$ is arbitrary, we have
\begin{align*}
     &\tclvf (\state_0,t_0)  - \min \bigl\{ \frac{\theta_1}{2}, \frac{\theta_2 (e^{\gamma \delta} - 1)}{2\gamma} \bigr\} \\
      \geq & \sup_{\dmap} \max \bigl\{ e^{\gamma (\tvar - \tinit_0)} \loss (\xi(s;t_0,x_0,t,\dmap)), \\
      & \hspace{6em} e^{\gamma \delta} \tclvf (\xi(t_0+\delta;x_0,t_0,\bar \ctrl (\cdot), \dmap[\bar \ctrl]),t_0+\delta) \bigr\} \\
      \geq & \tclvf(\state_0,t_0),
\end{align*}
where the last inequality is from~\eqref{eqn:TV-R-CLVF_DPP}. This is a contradiction, so~\eqref{eqn:TV_VS_sub} must hold, i.e., $\tclvf$ is a viscosity sub-solution.}

{Next, we examine the super-solution. Assume~\eqref{eqn:TV_VS_super} is wrong, then $\exists \theta_1, \theta_2 > 0$ s.t. one of the followings hold 
\begin{align}
    &\loss(\state_0) - \tclvf(\state_0,t_0) \geq \theta_1 > 0, \label{eqn:TV_VS_super1} \\
    & D_\tinit \psi(x_0,t_0) + \max _{\dstb \in \dset} \min _{ \ctrl \in \cset} D_x \psi(x_0,t_0) \cdot \dyn(\state_0, \ctrl, \dstb)  \notag \\ 
    &+ \gamma \tclvf (x_0,t_0) \geq \theta_2 > 0. \label{eqn:TV_VS_super2}
\end{align}
If~\eqref{eqn:TV_VS_super1} holds, then from~\eqref{eqn:TV_VS_l-V}, there exists $\delta _1 >0$ s.t. for any $\theta_1>0, \csig, \dmap$ and $\tvar \in [\tinit_0, \tinit_0+\delta _1]$
\begin{align*} 
    e^{\gamma (\tvar - \tinit_0)} \loss (\xi(s;t_0,x_0,u(\cdot),\dmap)) &\geq  \loss (\state_0) - \frac{\theta_1}{2} \\
    &\geq  \tclvf(\state_0,t_0) +\frac{\theta_1}{2}.
\end{align*}
Plug in the DPP~\eqref{eqn:TV-R-CLVF_DPP}, we have 
\begin{align*}
    \tclvf(\state_0,t_0) \geq &\sup_{\dmap} \inf_{\csig} \max_{s\in[\tinit_0, \tinit_0+ \delta]}e^{\gamma (\tvar - \tinit_0)} \loss (\xi(s;t_0,x_0,u(\cdot),\dmap)) \\
    \geq& \tclvf(\state_0,t_0) + \frac{\theta_1}{2},
\end{align*}
which is a contradiction. Therefore~\eqref{eqn:TV_VS_super1} cannot be true. 
If~\eqref{eqn:TV_VS_super2} holds, then from the same derivation of~\eqref{eqn:TV_VS_sub4}, there exists $\delta_2 > 0$, $\exists \dmap, \forall \csig$, s.t. $\forall \tvar \in [\tinit, \tinit+\delta_2]$
\begin{align*}
    &e^{\gamma \delta} \tclvf (\xi(t_0+\delta;t_0,x_0,\bar \ctrl (\cdot), \dmap[\bar \ctrl]), t_0+\delta)   \\ 
    \geq  &\psi(\state_0,t_0)  +\frac{\theta_2 (e^{\gamma \delta} - 1)}{2\gamma}
    =  \tclvf (\state_0,t_0)  +\frac{\theta_2 (e^{\gamma \delta} - 1)}{2\gamma}.  
\end{align*}
Again plug in DPP~\eqref{eqn:TV-R-CLVF_DPP}, we have
\begin{align*}
    \tclvf(\state_0,t_0) \geq &\sup_{\dmap} \inf_{\csig} e^{\gamma \delta} \tclvf (\xi(t_0+\delta;x,t,\bar \ctrl (\cdot), \dmap[\bar \ctrl]), t_0+ \delta)  \\
    \geq& \tclvf(\state_0,t_0) + \frac{\theta_2 (e^{\gamma \delta} - 1)}{2\gamma},
\end{align*}
which is a contradiction. Therefore~\eqref{eqn:TV_VS_super2} cannot be true. Combined,~\eqref{eqn:TV_VS_super} must hold, so $\clvf$ is a viscosity super-solution. }


{To prove the uniqueness, we apply the standard comparison principle. It should be noted that since we specify $\loss (x;p) = \|x-p\| - \minval $, the TV-R-CLVF~\eqref{eqn:finite_time_CLVF} is radially unbounded. However, given any finite $\tinit \leq 0$, its spatial growth is linearly bounded, i.e., there exists $C$ s.t. for all $x\in \mathbb R^n, t\leq 0$
\begin{align}\label{eqn:proof_TV_unique_1}
    \tclvf(x,t) \leq C (1+\|x\|). 
\end{align}
This is because $f$ is bounded, and the exponential term is bounded by $e^{\gamma t}$. Consider continuous functions $V_1$ and $V_2$ both Lipschitz (with Lipschitz constant $L_v$), radially unbounded, and satisfying~\eqref {eqn:proof_TV_unique_1}. Assume $V_1$ and $V_2$ are viscosity sub- and supersolutions of~\eqref{eqn:TV-R-CLVF-VI}, respectively. Define
\begin{align*}
    V_k(x,t) = V_1(x,t) + k(a\|x\|^2+c) 
\end{align*}
for constants $a,c,k>0$. Note that the growth of $V_k$ is quadratically bounded and Lipschitz, with Lipschitz constant $L_k$. Then $\loss(x)- V_k(x,t) < 0 $, and $V_k$ is the strict viscosity subsolution to 
\begin{align*}
    D_t \phi+ H(x,D_x \phi) + \gamma V_k \geq k,
\end{align*}
where $H$ is the Hamiltonian given by~\eqref{eqn:Hamitonian}. Recall that $H$ is bounded above and Lipschitz continuous. Denote 
\begin{align*}
    M_k = \sup_{[x,s]\in \mathbb R^n\times [t,0] }V_2(x,s)- V_k(x,s),
\end{align*}
and assume $M_k>0$. For any $\varepsilon,\eta,\delta > 0$, define:
\begin{align}
    \Phi_{\varepsilon,\delta} (x,y,s) =& V_k(x,s)-V_2(y,s)- \frac{\|x-y\| ^2}{2\varepsilon} \notag \\
    &-\delta(\|x\|^4+\|y\|^4).
\end{align}
Since $\Phi \rightarrow -\infty$ as $\|x\|+\|y\| \rightarrow \infty$, $\Phi$ attains its maximum at some state $(\bar x, \bar y, \bar s)$. Therefore, we have:
\begin{align*}
    \Phi(\bar x,\bar x, \bar s)+ \Phi(\bar y, \bar y, \bar s) \leq 2 \Phi(\bar x, \bar y , \bar s ),
\end{align*}
which gives us 
\begin{align*}
   \frac{ \|\bar x- \bar y \|^2} {\epsilon} &\leq V_k(\bar x , \bar s) - V_k(\bar y, \bar s) + V_2(\bar x , \bar s) - V_2(\bar y, \bar s) \\
   & \leq  (L_k+L_v) \| \bar x - \bar y\|. 
\end{align*}
This gives us 
\begin{align}\label{eqn:proof_uni_1}
     \|\bar x- \bar y \| \leq (L_k+L_v) \varepsilon,
\end{align}
i.e., $\bar x - \bar y \rightarrow 0$ as $\varepsilon \rightarrow 0$. Define test functions
\begin{align*}
    \phi(x,t) &= V_2(\bar y, \bar s) + \frac{\|x-\bar y\| ^2}{2\varepsilon} + \delta(\|x\|^4+\|\bar y\|^4),\\
    \psi (y,t) & =V_k(\bar x, \bar s) - \frac{\|\bar x- y\| ^2}{2\varepsilon} - \delta(\|\bar x\|^4+\| y\|^4),
\end{align*}
and we can check $V_k - \phi$ attains maximum at $(\bar x, \bar y, \bar s)$, $V_2 - \psi$ attains minimum at $(\bar x, \bar y, \bar s)$. Since $V_k$ is a strict viscosity subsolution, the following hold:
\begin{align}
    D_\tinit \phi(\bar x,\bar s) + H(\bar x, D_x\phi(\bar x,\bar s)) + \gamma V_k(\bar x,\bar s) \geq k .\label{eqn:proof_uni_sub2}
\end{align}
Since $V_2$ is a viscosity supersolution, from~\eqref{eqn:TV_VS_super}, both of the following hold
\begin{align}
    \loss (\bar y) - V_2(\bar y, \bar s) \leq 0 ,\label{eqn:proof_uni_sup1} \\
    D_\tinit \psi(\bar y,\bar s) +H(\bar y, D_x\psi (\bar y, \bar s)) + \gamma V_2(\bar y,\bar s) \leq 0. \label{eqn:proof_uni_sup2}
\end{align}
Then since $D_\tinit \phi(\bar x,\bar s)= D_\tinit \psi(\bar y,\bar s)$, from~\eqref{eqn:proof_uni_sub2} and~\eqref{eqn:proof_uni_sup2}, we have
\begin{align*}
    &\gamma (V_2 (\bar y,\bar s)- V_k(\bar x, \bar s)) \\
    \leq &- k + ( H(\bar x, D_x\phi(\bar x,\bar s)) - H(\bar y, D_x\psi (\bar y, \bar s))),\\
    \leq & -k + L_{Hx}(\bar x - \bar y) + L_{Hp}(D_x\phi(\bar x,\bar s)- D_x\psi(\bar y, \bar s)),
\end{align*}
where $L_{Hx}$ and $L_{Hp}$ are the Lipschitz constants of $H$ over $x$ and $p$ respectively. By the definition of $\phi,\psi$, we have
\begin{align*}
    D_x\phi(\bar x,\bar s)- D_x\psi(\bar y, \bar s) 
    = 3\delta (\|\bar x] \bar x - \|\bar y \| \bar y). 
\end{align*}
Combining with~\eqref{eqn:proof_uni_1}, as $k\rightarrow 0$, we can always pick $\varepsilon,\delta$ small enough s.t. 
\begin{align*}
    V_2 (\bar y,\bar s)- V_k(\bar x, \bar s) \leq - \frac{k}{2}. 
\end{align*} 
This contradicts the assumption that $M_k > 0$. Therefore, we proved $V_2\leq V_1$. Switching the role of $V_1$ and $V_2$ and following the identical process, we can prove $V_1 \leq V_2$. Therefore the TV-R-CLVF has a unique viscosity solution. }

\end{proof}

%% file: proofs/proof_thrm_CLVF_Lipschitz.tex
{Under the assumption that the system is exponentially stabilizable, $\forall x,y \in \mathcal \dom$, there exists constants $c>0, \alpha>0$ such that 
\begin{align}\label{eqn:proof_CLVF_LIP_1}
     \|\xi(s;t,x,u(\cdot),\dmap) - \xi(s;t,y,u(\cdot),\dmap)\| \notag \\
     \leq  e^{-\alpha (s-t)}c\|x-y\|.
\end{align}
Following the same process as in the proof of Proposition~\ref{prop: TV-R-CLVF_lipschitz}, we can pick $\lambda,u$ s.t.
\begin{align*}
     & \| J_\gamma (\tinit,x,u(\cdot),\dmap) - J_\gamma (\tinit,y,u(\cdot),\dmap) \|\\
    =&\| \max _{\tvar \in [\tinit, \thor]} e^{\gamma (\tvar-\tinit)} \ell\big(\xi(s;t,x,u(\cdot),\dmap)\big) - \\
    & \hspace{3em} \max _{\tvar \in [\tinit, \thor]} e^{\gamma (\tvar-\tinit)}\ell \big(\xi(s;t,y,u(\cdot),\dmap) \big) \|\\ 
    \leq & \max _{\tvar \in [\tinit, \thor]} \| e^{\gamma (\tvar-\tinit)} \ell\big(\xi(s;t,x,u(\cdot),\dmap)\big)  -\\
    & \hspace{3em} e^{\gamma (\tvar-\tinit)} \ell\big(\xi(s;t,y,u(\cdot),\dmap)\big) \| \\
    \leq & \max _{\tvar \in [\tinit, \thor]}  e^{\gamma (s-t)} \|  \ell\big(\xi(s;t,x,u(\cdot),\dmap)\big)  -\\
    & \hspace{3em} \ell\big(\xi(s;t,y,u(\cdot),\dmap)\big) \|. 
\end{align*}
Plugging in~\eqref{eqn:proof_CLVF_LIP_1}, we get 
\begin{align}
    & \| J_\gamma (\tinit,x,u(\cdot),\dmap) - J_\gamma (\tinit,y,u(\cdot),\dmap) \|  \notag \\
    \leq & \max _{\tvar \in [\tinit, \thor]}  e^{\gamma (s-t)}  e^{-\alpha (s-t)}c\|x-y\|  \notag \\
    = & \max _{\tvar \in [\tinit, \thor]} e^{(\gamma-\alpha) (s-t)} c\|x-y\| \notag \\
    \leq &c|x-y\|, 
\end{align}
where the last inequality is because $\gamma \leq \alpha$. Note that this inequality holds for all $t\leq 0$, and $\lambda$ and $u$ are aribitrary, we conclude that 
\begin{align*}
    \| \tclvf(x,t) - \tclvf(y,t) \| \leq c\| x-y \|.
\end{align*}
Since the Lipschitz constant is independent of $t$, we can pass it to the limit, which means
\begin{align*}
    \| \clvf(x) - \clvf(y) \| \leq c \|x-y\|. 
\end{align*}}

%% file: proofs/proof_prop_radiallyUnbounded.tex
{Let us assume there exists a constant $C>0$ s.t. as $\state \rightarrow \partial \dom$, $\clvf(\state) \leq C$. From the R-CLVF-DPP~\eqref{eqn:CLVF_DPP}, for any $\tinit \leq \tinit + \delta \leq 0$ both of the followings must hold: 
\begin{align*}
    & \clvf(\state) \geq e^{\gamma \delta} \clvf(z), \\
    & \clvf(\state) \geq \sup_{\dmap} \inf_{\csig} \max_{\tvar \in [\tinit, \tinit+\delta]} e^{\gamma (\tvar -\tinit)} \loss (\trajmap).
\end{align*}
From the first inequality, there exists some $\epsilon > 0$ s.t. 
\begin{align*}
    \clvf(\state) \geq &e^{\gamma \delta} \clvf(z) \\
     \geq & e^{\gamma \delta} (C- \epsilon) \\
     = & C - (C+e^{\gamma \delta} \epsilon - e^{\gamma \delta} C).
\end{align*}
For any constants $C, \epsilon$, we could find a $\delta$ large enough s.t. $C+e^{\gamma \delta} \epsilon - e^{\gamma \delta} C < 0$. This means 
\begin{align*}
    \clvf(\state) > C,
\end{align*}
which is a contradiction, and therefore such $C$ does not exists. }

%% file: proofs/proof_VdotleqV.tex
\begin{proof}
{Since the R-CLVF is only Lipschitz continuous, there exist points that are not differentiable. For those points, \cite{crandall1984some} showed that either a super-differential ($D^+\clvf (\state)$) or a sub-differential ($D^-\clvf (\state)$) exists, whose elements are called super-gradients and sub-gradients respectively. A function is differentiable at $x$ if $D^-\clvf (\state) = D^+\clvf (\state)$. Non-differentiable points only have a super-differential or sub-differential. At non-differentiable points, define $\dot \clvf (\state) = p\cdot f(x,u)$, where $p$ is either a sub-gradient or a super-gradient.}

{For non-differentiable points with super-differential, the corresponding solution is called a sub-solution, and
\begin{align*}
   &\max\biggl\{\loss(x) - \clvf (\state), \hspace{1em} \max _{\dstb \in \dset} \min _{\ctrl \in \cset}p^+ \cdot f(x, u,d) \\
  & \hspace{2em} + \gamma \clvf (\state) \biggl\} \leq 0,  \quad \forall p^+ \in D^+\clvf (\state).
\end{align*}
The maximum of the two terms is less or equal to 0, which implies both terms must be less or equal to 0: 
\begin{align*}
    \forall p^+ \in D^+\clvf (\state) \text{, } \max _{\dstb \in \dset} \min_{\ctrl \in \cset} p^+ \cdot f(x, u ,d ) \leq -\gamma \clvf (\state).
\end{align*}
This means for any super-gradients, there exists some control input, that will provide a sufficient decrease in the value along the trajectory. }

{When there exists sub-differential, we have: 
\begin{align*}
   &\max\biggl\{\loss(x) - \clvf (\state), \hspace{1em} \max _{\dstb \in \dset} \min _{\ctrl \in \cset}p^- \cdot f(x, u,d) \\
  & \hspace{2em} + \gamma \clvf (\state) \biggl\} \leq 0,  \quad \forall p^- \in D^-\clvf (\state).
\end{align*}
Since the R-CLVF is a Lipschitz continuous viscosity solution, we can directly apply Theorem 2.3 in \cite{frankowska1989hamilton}, we have
\begin{align*}
    \forall p^- \in D^-\clvf (\state) \text{ , } \max _{\dstb \in \dset}\min _{\ctrl \in \cset} p^- \cdot f(x, u ,d ) = -\gamma \clvf (\state).
\end{align*}
Combined, we get the desired inequality: $\dot \clvf  \leq - \gamma \clvf$ holds for all points in $\dom$.}
\end{proof}